\documentclass[12pt]{amsart}
\usepackage{graphicx,amssymb,upref}

\usepackage[shortlabels]{enumitem}
\usepackage{extpfeil} 

\usepackage{geometry}
\usepackage{tikz}
\usetikzlibrary{cd}
\usetikzlibrary{decorations.markings}
\usetikzlibrary{positioning} 
\usetikzlibrary{arrows} 

\pgfdeclarelayer{bg} 
\pgfsetlayers{bg,main} 

\usepackage{hyperref}
\hypersetup{colorlinks,urlcolor=blue,citecolor=blue,linkcolor=blue,linktocpage=false}

\usepackage[abbrev]{amsrefs}

\usepackage[all,2cell]{xy} \UseAllTwocells

\title[Eilenberg--MacLane mapping algebras]{Eilenberg--MacLane mapping algebras and higher distributivity up to homotopy} 

\author{Hans-Joachim Baues}
\address{Max-Planck-Institut f\"ur Mathematik\\
Vivatsgasse 7\\
53111 Bonn\\
Germany}
\email{baues@mpim-bonn.mpg.de}

\author{Martin Frankland}
\address{Universit\"at Osnabr\"uck\\
Institut f\"ur Mathematik\\
Albrechtstr. 28a\\
49076 Osnabr\"uck\\
Germany}
\email{martin.frankland@uni-osnabrueck.de}

\keywords{higher distributivity, distributivity up to homotopy, higher cohomology operation, Eilenberg-MacLane spectrum, Steenrod algebra, Kristensen derivation, homotopy invariant, A-infinity morphism, topological abelian group, mapping theory, mapping algebra.}

\subjclass[2010]{Primary: 55S20; Secondary: 55P20, 55S10, 18G55.}

\date{\today}

\tikzset{
 ->-/.style = {
  decoration = {markings, mark=at position #1 with {\arrow[scale=1.3]{>}}},
  postaction = {decorate}
 }
}

\newextarrow{\xrightrightarrow}{{20}{20}{20}{20}}{\bigRelbar\bigRelbar{\bigtwoarrowsleft\rightarrow\rightarrow}}

\numberwithin{equation}{section}
\numberwithin{figure}{section}

\theoremstyle{plain}
\newtheorem{thm}[equation]{Theorem}
\newtheorem{prop}[equation]{Proposition}
\newtheorem{cor}[equation]{Corollary}
\newtheorem{lem}[equation]{Lemma}

\theoremstyle{definition}
\newtheorem{defn}[equation]{Definition}
\newtheorem{nota}[equation]{Notation}

\newtheorem{rem}[equation]{Remark}
\newtheorem{ex}[equation]{Example}
\newtheorem{ques}[equation]{Question}

\newcommand{\bigcol}{3.8pc} 
\newcommand{\bigrow}{3.8pc} 

\newcommand{\F}{\mathbb{F}}

\newcommand{\Z}{\mathbb{Z}}

\newcommand{\cat}[1]{\mathcal{#1}} 
\newcommand{\ord}[1]{\left[ \mathbf{#1} \right]}
\newcommand{\steen}{\mathcal{A}} 

\newcommand{\De}{\Delta}
\newcommand{\de}{\delta}
\newcommand{\del}{\partial}
\newcommand{\ep}{\epsilon}
\newcommand{\Ga}{\Gamma}
\newcommand{\ga}{\gamma}
\newcommand{\io}{\iota}
\newcommand{\ka}{\kappa}
\newcommand{\la}{\lambda}
\newcommand{\Om}{\Omega}
\newcommand{\phy}{\varphi}
\newcommand{\Si}{\Sigma}
\newcommand{\si}{\sigma}
\newcommand{\te}{\theta}

\newcommand{\bu}{\bullet}
\newcommand{\inj}{\hookrightarrow}

\newcommand{\ral}{\xrightarrow} 
\newcommand{\Ra}{\Rightarrow}
\newcommand{\surj}{\twoheadrightarrow}
\newcommand{\tild}{\widetilde}

\newcommand{\dfn}{:=}
\newcommand{\inv}{\boxminus}

\newcommand{\ol}{\overline}
\newcommand{\op}{\oplus}
\newcommand{\ot}{\otimes}
\newcommand{\place}{\star}

\newcommand{\sm}{\wedge}
\newcommand{\sq}{\square \,}
\newcommand{\ul}{\underline}

\newcommand{\x}{\times}

\newcommand{\Ab}{\mathbf{Ab}}
\newcommand{\Gp}{\mathbf{Gp}}

\newcommand{\obs}[1]{\mathcal{O}(#1)} 
\newcommand{\Set}{\mathbf{Set}}
\newcommand{\Spec}{\mathbf{Spec}}
\newcommand{\Topp}{\mathbf{Top}}

\newcommand{\abs}[1]{\lvert #1 \rvert}

\DeclareMathOperator{\diag}{diag}
\DeclareMathOperator{\End}{End}
\DeclareMathOperator{\Ext}{Ext}
\DeclareMathOperator{\Hom}{Hom}

\DeclareMathOperator{\Ob}{Ob}

\newcommand{\cst}{\mathrm{cst}}
\newcommand{\id}{\mathrm{id}}

\newcommand{\ObIm}{\mathrm{ObIm}}

\newcommand{\proj}{\mathrm{proj}}
\newcommand{\Sing}{\mathrm{Sing}}
\newcommand{\sh}{\mathrm{sh}}

\newcommand{\Sq}{\mathrm{Sq}}

\newcommand{\Def}{\textbf} 

\newdir{->}{!/5pt/\dir{-}!/2.5pt/\dir{-}*!/-5pt/\dir2{>}}

\newcommand{\adjust}{0.4}
\definecolor{darkgreen}{rgb}{0,0.4,0}
\newcommand{\EdgeColor}{darkgreen}
\newcommand{\FaceColor}{purple}
\newcommand{\VertexColor}{black}
\newcommand{\VertexSize}{0.05em}
\newcommand{\DoubleArrow}{\draw [-implies, line width=0.16ex, double distance=0.55ex]} 

\begin{document} 
 
\begin{abstract}  
Primary cohomology operations, i.e., elements of the Steenrod algebra, are given by homotopy  classes of maps between Eilenberg--MacLane spectra. Such maps (before taking homotopy classes) form the topological version of the Steenrod algebra. Composition of such maps is strictly linear in one variable and linear up to coherent homotopy in the other variable. To describe this structure, we introduce a hierarchy of higher distributivity laws, and prove that the topological Steenrod algebra satisfies all of them. We show that the higher distributivity laws are homotopy invariant in a suitable sense. As an application of $2$-distributivity, we provide a new construction of a
derivation of degree $-2$ of the mod $2$ Steenrod algebra.
\end{abstract} 
\maketitle
\tableofcontents


\section{Introduction}

The elements of the Steenrod algebra are primary cohomology operations, which are given by homotopy classes of maps between Eilenberg--MacLane spectra. Richer information in contained in the %
\emph{mapping spaces} between Eilenberg--MacLane spectra, which form a topological 
version of the Steenrod algebra that encodes secondary and all higher order cohomology operations. 
One of the great successes of algebraic topology was the complete computation of the Steenrod algebra, and the study of its algebraic properties \cite{Milnor58} \cite{Adams58}.
In contrast, 
the algebraic nature of higher order cohomology operations remained %
lesser known. Only secondary operations were studied in detail, notably in \cite{Adams60}, \cite{Kristensen63}, \cite{KristensenM67eva}, \cite{Kristensen69}, \cite{Klaus01coc}, \cite{Harper02}, and \cite{Baues06}.

One of the difficulties with higher order cohomology operations is that they do not form an algebra. %
While composition of elements in the Steenrod algebra is bilinear, 
composition in the %
topological Steenrod algebra 
is not bilinear, but left linear (strictly) and right linear up to coherent homotopy.

In this paper, we introduce the notion of $n^{\text{th}}$ order distributivity (for $n \geq 0$), which is similar to the notion of $n^{\text{th}}$ order associativity or $n^{\text{th}}$ order commutativity. Stasheff described higher order associativity via $A_{\infty}$-spaces, based on associahedra \cite{Stasheff63i} \cite{MarklSS02}*{\S I.1.6, II.1.6}. Other polytopes have been used to describe homotopy coherent algebraic structure, such as permutahedra, which encode higher order commutativity \cite{Williams69}, or permuto-associahedra, which mix higher order associativity and commutativity \cite{Kapranov93}. A different (less strict) notion of higher distributivity is studied in \cite{Cranch10}*{\S 6} using distributahedra. 
The higher order distributivity that we consider turns out to be based on higher dimensional cubes. 

Our main result can be roughly stated as follows.

\begin{thm}
The %
topological Steenrod algebra satisfies infinitely many higher order coherent distributivity laws up to homotopy.
\end{thm}

\subsection*{Organization}

In Section~\ref{sec:EMmapping}, we introduce the notion of \emph{weakly bilinear mapping theory} (Definition~\ref{def:LeftLin}), which has an addition, a multiplication, left linearity, but might not be strictly right linear. The motivational example is the \emph{Eilenberg--MacLane mapping theory} $\cat{EM}$, given be finite products of Eilenberg--Maclane spectra and mapping spaces between them (Definition~\ref{def:EMMappingTh} and Proposition~\ref{pr:WeaklyBilinear}). 

In Section~\ref{sec:HigherDist}, we define higher distributivity via higher dimensional cubes (Definition~\ref{def:nDistrib}) and reformulate it as an inductive construction (Lemma~\ref{lem:nDistribInduc}). Our main result is that a weakly bilinear mapping theory is $\infty$-distributive, in a canonical way (Theorem~\ref{thm:InftyDistrib} and Proposition~\ref{pr:ExistNDistrib}).

Section~\ref{sec:GoodDistrib} studies additional properties that the distributivity data might satisfy. These are used in the proof of the main result.

In Section~\ref{sec:Kristensen}, we study applications of higher distributivity to the mod $2$ Steenrod algebra $\steen$. We recall how the Kristensen derivation $\ka \colon \steen \to \steen$ can be obtained from the $1$-distributivity of $\cat{EM}$; Kristensen's original construction used cochain operations. 
Using the $2$-distributivity of $\cat{EM}$, we %
provide a new construction of a 
derivation $\la \colon \steen \to \steen$ of degree $-2$ (Proposition~\ref{pr:LaDerivation}). We leave for future research an explicit algebraic formula for this derivation.

In Section~\ref{sec:HomotInvar}, we show that $n$-distributivity is homotopy invariant in some appropriate sense. More precisely, our notion of higher distributivity has some strict structure built in: a strict addition, a strictly associative multiplication, and left linearity holding strictly.  Corollary~\ref{cor:HomotInvar} says that $n$-distributivity is invariant with respect to Dwyer--Kan equivalences that preserve addition and multiplication strictly.

Appendix~\ref{sec:BFSpectra} collects for convenience some facts about Eilenberg--MacLane spectra.

\subsection*{Related work}

By applying the fundamental groupoid functor $\Pi_{1}$ to $\cat{EM}$, one obtains a groupoid-enriched theory $\Pi_{1} \cat{EM}$ which encodes \emph{secondary cohomology operations}, and was computed in \cite{Baues06}. %
One motivation for studying secondary operations was to compute the classical Adams differential $d_2$ \cite{BauesJ11}; c.f. Remark~\ref{rem:AdamsD2}. This paper does not address applications to the Adams spectral sequence. The paper \cite{BauesFrankland16} was about the Adams $d_3$, without using the additive structure of $\cat{EM}$. %

One of the steps in \cite{Baues06} %
was a structural result: replacing the track category $\Pi_1 \cat{EM}$ by a weakly equivalent DG-category over $\Z/p^2$, using the $1$-distributivity of $\cat{EM}$. We revisit that strictification result in \cite{BauesFranklandDGA}. The current paper does not address strictification. However, as groundwork towards the strictification problem for tertiary operations, we study the $2$-distributivity of $\cat{EM}$ in more detail in Section~\ref{sec:Kristensen}. 

\subsection*{Acknowledgments}

We thank the referee for their helpful comments. 
The second author thanks the Max-Planck-Institut f\"ur Mathematik Bonn for its generous hospitality, as well as Tobias Barthel, David Blanc, Ya\"el Fr\'egier, Lennart Meier, Fernando Muro, Irakli Patchkoria, Stefan Schwede, and Marc Stephan for useful conversations. 
The second author was partially funded by a grant of the DFG SPP 1786: Homotopy Theory and Algebraic Geometry.

\section{Notations and conventions}

Let $\Topp$ denote a convenient category of topological spaces, for instance compactly generated weakly Hausdorff spaces, so that internal hom objects $Y^X$ exist for all objects $X$ and $Y$ of $\Topp$. Let $\Topp_*$ denote the category of pointed spaces, with basepoints generically denoted by $0 \in X$. 

Enrichment in $\Topp_*$ will mean with respect to the smash product $X \sm Y$ as symmetric monoidal structure. 
Following the terminology of \cite{Kelly05}*{\S 1.2}, we call a $\Topp_*$-enriched category a $\Topp_*$-category for short, and likewise for functors. 
In a $\Topp_*$-category $\cat{C}$, we denote the composition map by
\[
\mu \colon \cat{C}(B,C) \sm \cat{C}(A,B) \to \cat{C}(A,C)
\]
and write $\mu(x,y) = xy$ for short. Equivalently, this can be described by the map
\[
\mu \colon \cat{C}(B,C) \x \cat{C}(A,B) \to \cat{C}(A,C)
\]
satisfying $\mu(0,y) = 0$ and $\mu(x,0) = 0$ for all $x$ and $y$.

\begin{nota} \label{nota:OmitObjects}
We write $x \in \cat{C}$ if $x$ is a \emph{morphism} in $\cat{C}$, i.e., $x \in \cat{C}(A,B)$ for some objects $A$ and $B$ of $\cat{C}$. %
From now on, whenever an expression such as $xy$ %
appears, it is understood that $x$ and $y$ must be composable, i.e., the target of $y$ is the source of $x$.
\end{nota}

\section{Eilenberg--MacLane mapping algebras} \label{sec:EMmapping}

In this section, we describe our main object of interest: a topologically enriched category $\cat{EM}$ that encodes cohomology operations of all higher orders.

\subsection{Topologically enriched Steenrod algebra}

\begin{nota} \label{nota:BFSpectra}
Let $\Spec$ denote the category of \emph{Bousfield--Friedlander spectra}, viewed as a $\Topp_*$-category. Details are given in Appendix~\ref{sec:BFSpectra}.
\end{nota}

\begin{nota}
Fix a prime number $p$ and let $\F_p$ denote the field of order $p$. Let $\steen$ denote the mod $p$ Steenrod algebra, with grading $\steen^* = H\F_p^* H\F_p$. For a spectrum $X$, the mod $p$ cohomology $H^*(X;\F_p)$ is viewed as a \emph{left} $\steen$-module.
\end{nota}

\begin{defn} \label{def:EMMappingTh}
For every $n \in \Z$, let $K_n$ be %
a spectrum of the homotopy type of $\Si^n H\F_p$ 
as in Corollary~\ref{cor:GoodModelFiniteProd}. The mod $p$ \Def{Eilenberg--MacLane mapping theory}
$\cat{EM}$ is the full subcategory of $\Spec$ consisting of finite products
\[
A = K_{n_1} \x \ldots \x K_{n_k}.
\]
\end{defn}

The mapping theory $\cat{EM}$ is a topological refinement of the Steenrod algebra $\steen$. In fact, %
the category $\pi_0 \cat{EM}$ of path components of $\cat{EM}$ is equivalent to the opposite of the category of finitely generated free $\steen$-modules, so that $\pi_0 \cat{EM}$ is the theory of $\steen$-modules.

\begin{prop} \label{pr:WeaklyBilinear}
The $\Topp_*$-category $\cat{T} \dfn \cat{EM}$ has the following properties.
\begin{enumerate}
\item \label{item:FiniteProd} It has finite products.
\item \label{item:TopAbGp} All mapping spaces $\cat{T}(A,B)$ are topological abelian groups, with the basepoint $0$ being the additive identity.
\item \label{item:LeftLin} Composition is left linear, i.e., satisfies $(x+x')y = xy + x'y$.
\item \label{item:TrivFib} For all objects $A,B,Z$ of $\cat{T}$ the map
\[
\xymatrix @C=3pc {
\cat{T}(A \x B, Z) \ar[r]^-{(i_A^*, i_B^*)} & \cat{T}(A,Z) \x \cat{T}(B,Z) \\
}
\]
is a trivial fibration, i.e., a Serre fibration and a weak equivalence. Here, $i_A \colon A \to A \x B$ and $i_B \colon B \to A \x B$ denote the inclusion maps given by $i_A = (1_A, 0)$ and $i_B = (0,1_B)$. %
\end{enumerate}
\end{prop}

\begin{proof}
By construction, $\cat{T}$ has finite products, which are the same as in the ambient category $\Spec$. Since each object $A$ of $\cat{T}$ is an abelian group object in $\Spec$, the mapping space $\Spec(X,A)$ is a topological abelian group, with pointwise addition in the target $A$, which makes composition left linear.

For objects $A$ and $B$ in $\cat{T}$, the natural map $\io \colon A \vee B \to A \x B$ is a cofibration in $\Spec$, by Lemma~\ref{lem:WeakPushoutProd}. Since $\Spec$ is a simplicial model category, the restriction map
\[
\xymatrix{
\ul{\Spec}(A \x B, Z) \ar[dr]_{(i_A^*, i_B^*)} \ar[r]^-{\io^*} & \ul{\Spec}(A \vee B, Z) \ar[d]^{\cong} \\
& \ul{\Spec}(A,Z) \x \ul{\Spec}(B,Z) \\
}
\]
is a Kan fibration for any fibrant object $Z$ in $\Spec$, in particular for any object in $\cat{T}$. Moreover, the map $\io \colon A \vee B \ral{\sim} A \x B$ is a weak equivalence. Hence, the restriction map $\io^*$ is a weak equivalence \cite{GoerssJ09}*{Lemma~II.4.2}. To conclude, use the fact that the geometric realization of a Kan fibration is a Serre fibration \cite{GoerssJ09}*{Theorem~10.10}.
\end{proof}

\begin{rem}
Consider the full subcategory $\ol{\cat{EM}}$ of $\Spec$ with objects the bounded below degreewise finite products $K = \prod_i K_{n_i}$. This category $\ol{\cat{EM}}$ 
also has the properties listed in Proposition~\ref{pr:WeaklyBilinear}.
It appears notably in the context of $H\F_p$-based Adams resolutions \cite{BauesFrankland16}*{\S 7}. 
\end{rem}

Let us give names to the features appearing in the proposition. For our main result, we will not use commutativity or additive inverses in the mapping spaces $\cat{T}(A,B)$. However, we will \emph{still use additive notation}.%

\begin{defn} \label{def:LeftLin}
A $\Topp_*$-category $\cat{T}$ is called:
\begin{enumerate}[(a)]
\item a \Def{mapping theory} if it is small and has finite products.%
\item \Def{left linear} if all mapping spaces $\cat{T}(A,B)$ are topological monoids, with the basepoint $0$ being the additive identity, and composition is left linear.
\end{enumerate}
A mapping theory is \Def{weakly bilinear} if it is left linear and satisfies property~\eqref{item:TrivFib} from Proposition~\ref{pr:WeaklyBilinear}.

A \Def{morphism} of left linear $\Topp_*$-categories is a $\Topp_*$-functor $F \colon \cat{S} \to \cat{T}$ such that for all objects $A,B$ of $\cat{S}$, the induced map
\[
F \colon \cat{S}(A,B) \to \cat{T}(FA,FB)
\]
is a monoid homomorphism (i.e., preserves addition strictly).%
\end{defn}

\subsection{Topologically enriched cohomology}

\begin{defn}
Given a cofibrant spectrum $X$, the mod $p$ \Def{Eilenberg--MacLane mapping algebra} of $X$ consists of $\cat{EM}$ together with the functor $F_X \colon \cat{EM} \to \Topp_*$ represented by $X$, given by $F_X(A) := \Spec(X,A)$. 
\end{defn}

Note that $F_X$ is a topological refinement of the cohomology of $X$ as an $\steen$-module, which is recovered as
\[
\pi_0 F_X(K_n) = [X,K_n] = H^n (X;\F_p).
\]
The notion of a mapping algebra can be described formally as follows; 
variants appear in \cite{BauesB11}*{\S 8}, \cite{BadziochBD14}*{\S 1}, and \cite{BlancS14}.

\begin{defn}
\begin{enumerate}
\item A \Def{model} of a mapping theory $\cat{T}$ is a $\Topp_*$-functor $F \colon \cat{T} \to \Topp_*$ which preserves finite products (strictly).
\item A \Def{mapping algebra} $(\cat{T},F)$ consists of a mapping theory $\cat{T}$ together with a model $F$ of $\cat{T}$.
\end{enumerate}
\end{defn}

\begin{rem} \label{rem:Placeholder}
The data of a mapping algebra $(\cat{T}, F)$ can be encoded into a $\Topp_*$-category $\cat{T} \{F\}$, whose objects are those of $\cat{T}$ plus a \Def{distinguished object} $\place$, and mapping spaces are given by
\[
\cat{T} \{F\} (A,B) = \begin{cases}
\cat{T}(A,B) &\text{if } A,B \in \Ob \cat{T} \\
F(B) &\text{if } A = \place, B \in \Ob \cat{T} \\
\ast &\text{if } A \in \Ob \cat{T}, B = \place \\
\{0, 1_{\place} \} &\text{if } A = B = \place. \\
\end{cases}
\]
Since $F$ preserves finite products, the product $A \x B$ in $\cat{T}$ is still a product in $\cat{T} \{F\}$. However, $\cat{T} \{F\}$ does not have products involving the distinguished object $\place$. 

Using this construction, statements about mapping theories will have analogues about mapping algebras. The arguments apply as long as we never map \emph{into} the distinguished object $\place$. Likewise, the notion of left linearity in Definition~\ref{def:LeftLin} has a straightforward analogue for mapping algebras.
\end{rem}

\begin{rem}\label{rem:AdamsD2}
Denote the Eilenberg--MacLane mapping algebra of a spectrum $X$ by $\cat{EM} \{ X \} \dfn \cat{EM} \{ F_X \}$. %
In \cite{BauesB15}, it was shown how the classical Adams spectral sequence
\[
E_2^{s,t} = \Ext_{\steen}^{s,t} \left( H^*(X;\F_p) , \F_p \right) \Ra \pi_{t-s} X^{\wedge}_p
\]
can be derived from $\cat{EM} \{ X \}$, more specifically, how the Postnikov section $P_n \cat{EM}\{ X \}$ determines the spectral sequence up to the $E_{n+2}$ term. In particular, the \emph{secondary cohomology} $\Pi_{1} \cat{EM} \{ X \}$ of the spectrum $X$ determines the $E_3$ term of the spectral sequence. 
\end{rem}

\subsection{A generalization}

For the record, we extract a more general statement from the proof of Proposition~\ref{pr:WeaklyBilinear}.

\begin{prop} \label{pr:SimpModelCat}
Let $\cat{C}$ be a pointed simplicial model category \cite{Quillen67}*{\S II.2} \cite{GoerssJ09}*{\S II.3}. View $\cat{C}$ as a $\Topp_*$-category by taking the geometric realization of the simplicial mapping spaces.
Assume that $\cat{C}$ satisfies the following: For any cofibrant objects $X$ and $Y$, the map $\io \colon X \vee Y \to X \x Y$ is a cofibration. %

Let $\cat{G} \subseteq \Ob \cat{C}$ be a set of monoid %
objects in $\cat{C}$ %
which are fibrant and cofibrant. Assume moreover that for any $A, B$ in $\cat{G}$, the map $A \vee B \to A \x B$ is a weak equivalence. Then the full subcategory $\cat{T}_{\cat{G}}$ of $\cat{C}$ consisting of finite products of objects of $\cat{G}$ is a weakly bilinear mapping theory.

Note that every object in $\cat{T}_{\cat{G}}$ is fibrant and cofibrant as an object in $\cat{C}$, so that the mapping spaces $\cat{C}(A,B)$ are derived mapping spaces.

\end{prop}

\begin{ex}
Proposition~\ref{pr:WeaklyBilinear} also holds for the integral Eilenberg--MacLane mapping theory, using $\Z$ instead of $\F_p$ as coefficient group.
\end{ex}

\section{Higher distributivity} \label{sec:HigherDist}

\subsection{Cubes in a space}

In this subsection, we fix some notation about cubes in a space or a topologically enriched category.

\begin{defn} \label{def:Cubes}
Let $X$ be a topological space.

An \Def{$n$-cube} in $X$ is a map $\ga \colon I^n \to X$, where $I = [0,1]$ is the unit interval. For example, a $0$-cube in $X$ is a point of $X$, and a $1$-cube in $X$ is a path in $X$. In this case, we also denote $\ga$ as an arrow $\ga \colon \ga(0) \to \ga(1)$.

An \Def{$n$-track} in $X$ is a homotopy class, relative to the boundary $\del I^n$, of an $n$-cube. If $\ga \colon I^n \to X$ is an $n$-cube in $X$, denote by $\{ \ga \}$ the corresponding $n$-track in $X$, namely the homotopy class of $\ga$ rel $\del I^n$.
\end{defn}

\begin{defn}
Let $X$ be a pointed space, with basepoint $0 \in X$. The constant map $0 \colon I^n \to X$ with value $0 \in X$ is called the \Def{trivial $n$-cube}.

\end{defn}

The equality $I^{m+n} = I^m \x I^n$ allows us to define an operation on cubes.

\begin{defn} \label{def:ProdCubes}
Let $\mu \colon X \x X' \to X''$ be a composition map in a $\Topp_*$-category $\cat{C}$. For $m,n \geq 0$, consider cubes
\[
\begin{cases}
a \colon I^m \to X \\
b \colon I^n \to X'. \\
\end{cases}
\]
The \Def{$\ot$-composition} of $a$ and $b$ is the $(m+n)$-cube $a \ot b$ defined as the composite
\begin{equation} \label{TensorCubes}
a \ot b \colon I^{m+n} = I^m \x I^n \ral{a \x b} X \x X' \ral{\mu} X''. 
\end{equation}

For $m=n$, the \Def{pointwise composition} of $a$ and $b$ is the $n$-cube defined as the composite
\begin{equation} \label{PointwiseCubes}
ab \colon I^{n} \ral{(a,b)} X \x X' \ral{\mu} X''. 
\end{equation}
The pointwise composition is the restriction of the $\ot$-composition along the diagonal:
\[
\xymatrix{
I^n \ar[r]^-{\De} \ar@/_2pc/[rr]_{ab} & I^n \x I^n \ar[r]^-{a \ot b} & X''.
}
\]
Similarly, let $X$ be a topological %
monoid. The \Def{external addition} of cubes $a \colon I^m \to X$ and $b \colon I^n \to X$ is the $(m+n)$-cube $a \op b$ defined as the composite
\begin{equation} \label{ExtAddCubes}
a \op b \colon I^{m+n} = I^m \x I^n \ral{a \x b} X \x X \ral{+} X. 
\end{equation}
For $m=n$, the \Def{pointwise addition} of $a$ and $b$ is the $n$-cube defined as the composite
\begin{equation} \label{PointwiseAdd}
a+b \colon I^{n} \ral{(a,b)} X \x X \ral{+} X. 
\end{equation}
As an abuse of notation, we will also write $xy := x \ot y$ and $x+y := x \op y$ if $\deg(x) = 0$ or $\deg(y) = 0$ holds.
\end{defn}

\begin{lem} \label{lem:LeftLinCubes}
Let $\cat{T}$ be a left linear $\Topp_*$-category. %
Then the $\ot$-composition with a $0$-cube is left linear with respect to the external addition. More precisely, consider cubes in a mapping space in $\cat{T}$
\[
\begin{cases}
a \colon I^m \to \cat{T}(A,B) \\
b \colon I^n \to \cat{T}(A,B) \\
\end{cases}
\]
and a map $x \colon X \to A$. %
Then the equality
\[
(a \op b) x = (a x) \op (b x)
\]
holds, where both sides are $(m+n)$-cubes in $\cat{T}(X,B)$.
\end{lem}

\subsection{Definition of higher distributivity}

Before defining higher distributivity in general, let us look at some low-dimensional cases.

\begin{defn} \label{def:2Distrib}
Let $\cat{T}$ be a left linear $\Topp_*$-category. 
Then $\cat{T}$ is called \Def{$1$-distributive} if for all $a,x,y \in \cat{T}$, there is a path
\[
\begin{tikzpicture}[scale=2]
\draw [->-=0.5] (0,0) -- (2,0);
\draw [fill=\VertexColor] (0,0) circle (\VertexSize);
\node [below left] at (0,0) {$a(x + y)$};
\draw [fill=\VertexColor] (2,0) circle (\VertexSize);
\node [below right] at (2,0) {$ax + ay$.};
\node [below, text=\EdgeColor] at (1,0) {$\phy_a^{x,y}$};
\end{tikzpicture}
\]
in $\cat{T}$. A choice of such paths for $a,x,y \in \cat{T}$ is denoted $\phy^1 = \left\{ \phy_a^{x,y} \mid a,x,y \in \cat{T} \right\}$ and is called a \Def{$1$-distributor} for $\cat{T}$. Per Notation~\ref{nota:OmitObjects}, here we mean for all $a,x,y \in \cat{T}$ such that $a(x+y)$ is defined. Also, $\phy^1$ is required to be continuous in the inputs $a,x,y$. More precisely, for all objects $X,A,B$ of $\cat{T}$, the map
\[
\xymatrix @R=0.1pc {
\cat{T}(A,B) \x \cat{T}(X,A)^2 \ar[r]^-{\phy^1} & \cat{T}(X,B)^{I} \\
(a,x,y) \ar@{|->}[r] & \phy_a^{x,y} \\
}
\]
is continuous.

Next, $\cat{T}$ is called \Def{$2$-distributive} if it admits a $1$-distributor $\phy^1$ such that for all $a,x,y,z \in \cat{T}$, the map $\del I^2 \to \cat{T}$ defined as in the diagram of paths
\[
\begin{tikzpicture}[scale=1.5]
\draw [->-=0.5] (0,0) -- (0,2);
\draw [->-=0.5] (0,2) -- (2,2);
\draw [->-=0.5] (0,0) -- (2,0);
\draw [->-=0.5] (2,0) -- (2,2);
\draw [fill=\VertexColor] (0,0) circle (\VertexSize);
\node [below left] at (0,0) {$a(x + y + z)$};
\draw [fill=\VertexColor] (0,2) circle (\VertexSize);
\node [above left] at (0,2) {$a(x + y) + az$};
\draw [fill=\VertexColor] (2,0) circle (\VertexSize);
\node [below right] at (2,0) {$ax + a(y + z)$.};
\draw [fill=\VertexColor] (2,2) circle (\VertexSize);
\node [above right] at (2,2) {$ax + ay + az$};
\node [below, text=\EdgeColor] at (1,0) {$\phy_a^{x, y + z}$};
\node [above, text=\EdgeColor] at (1,2) {$\phy_a^{x, y} + az$};
\node [left, text=\EdgeColor] at (0,1) {$\phy_a^{x + y, z}$};
\node [right, text=\EdgeColor] at (2,1) {$ax + \phy_a^{y, z}$};
\end{tikzpicture}
\]
admits an extension $\phy_a^{x,y,z} \colon I^2 \to \cat{T}$. A choice of such $2$-cubes for $a,x,y,z \in \cat{T}$ is denoted \[
\phy^2 = \left\{ \phy_a^{x,y,z} \mid a,x,y,z \in \cat{T} \right\}
\]
and is called a \Def{$2$-distributor} for $\cat{T}$, \Def{based} on the $1$-distributor $\phy^1$. As before, the $2$-distributor $\phy^2$ is required to be continuous in the inputs $a,x,y,z \in \cat{T}$.

\end{defn}

\begin{defn} \label{def:nDistrib}
A left linear $\Topp_*$-category $\cat{T}$ is called \Def{$n$-distributive} if there are collections of cubes $\phy^0, \phy^1, \ldots, \phy^n$, where
\[
\phy^m = \{ \phy_a^{x_0, \ldots, x_m} \mid a, x_0, \ldots, x_m \in \cat{T} \}
\]
is a collection of $m$-cubes $\phy_a^{x_0, \ldots, x_m} \colon I^m \to \cat{T}$, satisfying the following:
\begin{itemize}
\item $\phy^0$ is the collection of $0$-cubes $\phy_a^x = ax$. %
\item For $1 \leq m \leq n$, the following boundary conditions hold:
\begin{equation} \label{eq:Boundary}
\resizebox{1.0\textwidth}{!}{
$\phy_a^{x_0, \ldots, x_m}(t_1, \ldots, t_m) = \begin{cases}
\phy_a^{x_0, \ldots, x_{j-1} + x_j, \ldots, x_m}(t_1, \ldots, \widehat{t_j}, \ldots, t_m) &\text{if } t_j = 0 \\
\phy_a^{x_0, \ldots, x_{j-1}}(t_1, \ldots, \ldots, t_{j-1}) \op \phy_a^{x_j, \ldots, x_m}(t_{j+1}, \ldots, t_m) &\text{if } t_j = 1. \\
\end{cases}$
}
\end{equation}
\end{itemize}
Such a collection $\phy^n$ of $n$-cubes in $\cat{T}$ is called an \Def{$n$-distributor} for $\cat{T}$, \Def{based} on the $(n-1)$-distributor $\phy^{n-1}$. The $n$-distributor $\phy^n$ is required to be continuous in the inputs $a, x_0, \ldots, x_n \in \cat{T}$. More precisely, for all objects $X,A,B$ of $\cat{T}$, the map
\[
\xymatrix @R=0.1pc {
\cat{T}(A,B) \x \cat{T}(X,A)^{n+1} \ar[r]^-{\phy^n} & \cat{T}(X,B)^{I^n} \\
(a,x_0, \ldots, x_n) \ar@{|->}[r] & \phy_a^{x_0, \ldots, x_n} \\
}
\]
is continuous. Note that the case $n=2$ agrees with Definition~\ref{def:2Distrib}.

The case $n = \infty$ is allowed: An \Def{$\infty$-distributor} for $\cat{T}$ is a sequence $\{ \phy^m \}_{m \geq 0}$ of families of cubes satisfying the above conditions, for all $m \geq 0$. 
\end{defn}

This definition is closely related to the notion of $A_{\infty}$ morphisms. The connection is described more precisely in Section~\ref{sec:Pushforward}.

\begin{rem} \label{rem:HighestDim}
In the data of an $n$-distributor, we only retain the $n$-cubes $\phy^n$, since the lower dimensional cubes $\phy^k$ (for $0 \leq k \leq n-1$) are determined by the boundary condition
\begin{align*}
\phy_a^{x_0, \ldots, x_{m-1}, 0}(t_1, \ldots, t_{m-1}, 0) &= \phy_a^{x_0, \ldots, x_{m-1} + 0}(t_1, \ldots, t_{m-1}) \\
&= \phy_a^{x_0, \ldots, x_{m-1}}(t_1, \ldots, t_{m-1}).
\end{align*}
\end{rem}

\subsection{Inductive construction of distributors}

Applying the boundary conditions of $\phy^n$  in Equation~\eqref{eq:Boundary} repeatedly, one can find the restriction $\phy^n \vert_{C} \colon C \to \cat{T}$ to any face $C \subseteq I^n$ of dimension less than $n$. We now describe this formula explicitly.

\begin{nota} \label{nota:CubesFaces}
Let $n \geq 1$. The cells (or subcubes) of the cube $I^n$ consist of the subsets $C \subseteq I^n$ of the form
\[
X_1 \x \ldots \x X_n \subseteq I^n
\]
where each $X_i$ is $\{0\}$ or $\{1\}$ or $I = [0,1]$. These cells are in bijection with functions $\si \colon \{ 1, \ldots, n \} \to \{ 0, 1, I \}$, which we call \Def{codes} for convenience, and 
sometimes write as a sequence of values $\left( \si(1), \ldots, \si(n) \right)$. 
Denote by $C_{\si}$ the cell corresponding to $\si$.%
\end{nota}

The subcube $C_{\si}$ has dimension $\abs{ \si^{-1}(I) }$, the number of free coordinates. Because of this, we also denote $\dim \si \dfn \abs{ \si^{-1}(I) }$. The entire cube is $I^n = C_{(I, \ldots, I)}$, while its boundary is the union of proper faces
\[
\del I^n = \bigcup_{\substack{\si \in \{0,1,I\}^n\\ \si \neq (I, \ldots, I)}} C_{\si}.
\] 
The vertices of the cube $I^n$ will parametrize different ways of going from $a(x_0 + \ldots + x_n)$ to $ax_0 + \ldots + ax_n$ by distributing the product over the sums. These expressions contain $n$ symbols $+$, and we will interpret the value $\si(i)$ as telling whether the $i^{\text{th}}$ symbol $+$ has been brought outside, with $1$ or $0$ meaning yes or no, respectively. For example, $(1,0,0,1)$ corresponds to $ax_0 + a(x_1 + x_2 + x_3) + ax_4$. 

\begin{ex}
Consider the case $n=2$. 
The subcubes of the cube $I^2$ are assigned distributors as in Figure~\ref{fig:Distrib2Cube}.

\begin{figure}[ht]
\begin{tikzpicture}[scale=1.8]
\draw [->-=0.5] (0,0) -- (0,2);
\draw [->-=0.5] (0,2) -- (2,2);
\draw [->-=0.5] (0,0) -- (2,0);
\draw [->-=0.5] (2,0) -- (2,2);
\draw [fill=\VertexColor] (0,0) circle (\VertexSize);
\node [below left] at (0,0) {$a(x_0 + x_1 + x_2)$};
\draw [fill=\VertexColor] (0,2) circle (\VertexSize);
\node [above left] at (0,2) {$a(x_0 + x_1) + ax_2$};
\draw [fill=\VertexColor] (2,0) circle (\VertexSize);
\node [below right] at (2,0) {$ax_0 + a(x_1 + x_2)$.};
\draw [fill=\VertexColor] (2,2) circle (\VertexSize);
\node [above right] at (2,2) {$ax_0 + ax_1 + ax_2$};
\node [below, text=\EdgeColor] at (1,0) {$\phy_a^{x_0,x_1 + x_2}$};
\node [above, text=\EdgeColor] at (1,2) {$\phy_a^{x_0,x_1} + ax_2$};
\node [left, text=\EdgeColor] at (0,1) {$\phy_a^{x_0 + x_1, x_2}$};
\node [right, text=\EdgeColor] at (2,1) {$ax_0 + \phy_a^{x_1, x_2}$};
\node [text=\FaceColor] at (1,1) {$\phy_a^{x_0, x_1, x_2}$};
\end{tikzpicture}
\caption{The distributors assigned to subcubes of a $2$-cube.}
\label{fig:Distrib2Cube}
\end{figure}
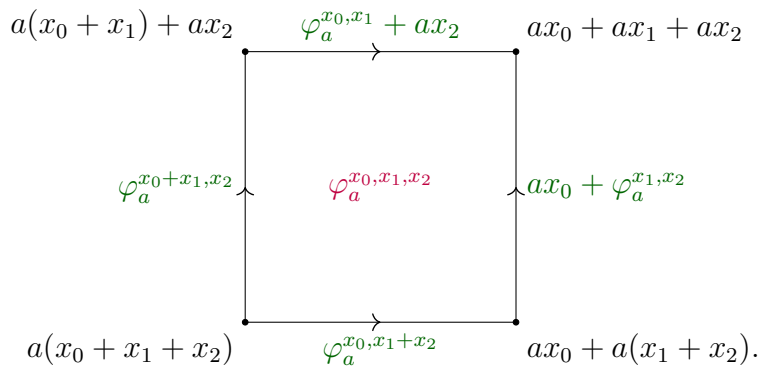
If a $1$-distributor $\phy^1$ is given, then this defines the obstruction map $\obs{\phy^1} \colon \del I^2 \to \cat{T}$ of Definition~\ref{def:nObstruc}.

Note that the product $ax_i$ is itself the $0$-distributor $\phy_a^{x_i}$. Also note that the pointwise sum of $0$-cubes $ax_0 + ax_1$ is also an external sum $ax_0 \oplus ax_1$, and the expression $\phy_a^{x_0,x_1} + ax_2$ is shorthand notation for the external sum $\phy_a^{x_0,x_1} \oplus ax_2$.
\end{ex}

\begin{defn} \label{def:FaceFormula}
Let $n \geq 1$ and let $\phy^{m}$ be an $m$-distributor for $\cat{T}$. Let $C_{\si} \subseteq I^n$ be a proper subcube of dimension $\dim \si \leq m$. The \Def{face formula} associates to each $a, x_0, \ldots, x_n \in \cat{T}$ (such that the expression $a(x_0 + \ldots + x_n)$ is defined) a map
\[
\phy^{m}[\si]_a^{x_0, \ldots, x_n} \colon C_{\si} \to \cat{T}
\]
as follows.

\emph{Step 1}: By convention, extend the code $\si$ by $\si(0) = 1$. Let
\[
\{ 0, 1, \ldots, n \} = J_0 \sqcup J_1 \sqcup \ldots \sqcup J_t
\]
be the partition into intervals satisfying
\begin{equation} \label{eq:OnlyOne1}
\si \vert_{J_k}(i) = \begin{cases}
1 &\text{if } i = \min J_k \\
0 \text{ or } I &\text{if } i \neq \min J_k. \\ 
\end{cases}
\end{equation}
In particular, $t = \abs{\si^{-1}(1)}$ is the number of $1$'s in the code. We define
\[
\phy^{m}[\si] \dfn \phy^{m} \left[ \si\vert_{J_0} \right] \op \phy^{m} \left[ \si\vert_{J_1} \right] \op \ldots \op \phy^{m} \left[ \si\vert_{J_t} \right].
\]

\emph{Step 2}: For an interval of integers $J$ satisfying Equation~\eqref{eq:OnlyOne1}, let
\[
J = K_0 \sqcup K_1 \sqcup \ldots \sqcup K_d
\]
be the partition into intervals satisfying
\[
\si \vert_{K_l}(i) = \begin{cases}
1 \text{ or } I &\text{if } i = \min K_l \\
0 &\text{if } i \neq \min K_l. \\ 
\end{cases}
\]
In particular, $d = \abs{\si^{-1}(I) \cap J}$ is the number of $I$'s in $\si\vert_J$. We define
\[
\phy^{m} \left[ \si\vert_J \right] \dfn \phy_{a}^{x_{K_0}, x_{K_1}, \ldots, x_{K_d}}
\]
where we denote $x_K \dfn \sum_{k \in K} x_k$ for any set of integers $K$, keeping the order of the inputs $x_k$.
\end{defn}

\begin{ex}
With $n = 8$ and the code $\si = 0I11I00I$, we have:
\begin{align*}
J_0 &= \{ 0,1,2 \} = \{ 0,1 \} \sqcup \{ 2 \} \\
J_1 &= \{ 3 \} \\
J_2 &= \{ 4,5,6,7,8 \} = \{ 4 \} \sqcup \{ 5, 6, 7 \} \sqcup \{ 8 \}.
\end{align*}
The face formula yields
\begin{align*}
\phy^{m}[\si] &= \phy^{m} \left[ \si\vert_{J_0} \right] \op \phy^{m} \left[ \si\vert_{J_1} \right] \op \phy^{m} \left[ \si\vert_{J_2} \right] \\
&= \phy_a^{x_0 + x_1, x_2} \op \phy_a^{x_3} \op \phy_a^{x_4, x_5 + x_6 + x_7, x_8}
\end{align*}
which is defined as long as $m \geq 3 = \dim \si$ holds.
\end{ex}

\begin{lem} \label{lem:FaceCompat}
A (continuous) collection $\phy^n$ of cubes $I^n \to \cat{T}$ is an $n$-distributor if and only if it satisfies $\phy^{n} \vert_{C_{\si}} = \phy^{n} [\si]$ for every subcube $C_{\si} \subseteq I^n$.
\end{lem}

\begin{defn} \label{def:nObstruc}
Let $\phy^{n-1}$ be an $(n-1)$-distributor for $\cat{T}$. The \Def{obstruction to $n$-distributivity} %
is the collection of maps
\[
\obs{\phy^{n-1}}_a^{x_0, \ldots, x_n} \colon \del I^n \to \cat{T}
\]
defined by their restriction to each face $C_{\si} \subset \del I^n$: 
\[
\obs{\phy^{n-1}} \vert_{C_{\si}} \dfn \phy^{n-1} [\si].
\]
It follows from the face formula that $\phy^{n-1}[\si]$ and $\phy^{n-1}[\si']$ agree on the intersection $C_{\si} \cap C_{\si'}$, so that the map $\obs{\phy^{n-1}} \colon \del I^n \to \cat{T}$ is well-defined.
\end{defn}

\begin{ex}
Consider the case $n=3$. 
The subcubes are assigned distributors as in Figure~\ref{fig:Distrib3Cube}.

\begin{figure}[ht]
\hspace*{-1.3cm}
\begin{tikzpicture}[scale=1.7]
\draw [fill=\VertexColor] (0,0) circle (\VertexSize);
\node [below left] [fill=white] at (0,0) {$a(x_0 + x_1 + x_2 + x_3)$};
\draw [fill=\VertexColor] (0,4) circle (\VertexSize);
\node [left] [fill=white] at (0,4) {$a(x_0 + x_1 + x_2) + ax_3$};
\draw [fill=\VertexColor] (4,0) circle (\VertexSize);
\node [below right] [fill=white] at (4,0) {$ax_0 + a(x_1 + x_2 + x_3)$};
\draw [fill=\VertexColor] (4,4) circle (\VertexSize);
\node [right] [fill=white] at (4,4) {$ax_0 + a(x_1 + x_2) + ax_3$};
\begin{pgfonlayer}{bg}
 \draw [fill=\VertexColor] (2,2) circle (\VertexSize);
 \node [left] [fill=white] at (2,2) {$a(x_0 + x_1) + a(x_2 + x_3)$};
\end{pgfonlayer}
\draw [fill=\VertexColor] (2,6) circle (\VertexSize);
\node [above left] [fill=white] at (2,6) {$a(x_0 + x_1) + ax_2 + ax_3$};
\draw [fill=\VertexColor] (6,2) circle (\VertexSize);
\node [right] [fill=white] at (6,2) {$ax_0 + ax_1 + a(x_2 + x_3)$};
\draw [fill=\VertexColor] (6,6) circle (\VertexSize);
\node [above right] [fill=white] at (6,6) {$ax_0 + ax_1 + ax_2 + ax_3$};
\draw [->-=0.5] (0,0) -- (4,0);
\node [text=\EdgeColor, fill=white] at (2,0-0.25) {$\phy_a^{x_0, x_1 + x_2 + x_3}$};
\draw [->-=0.4] (0,4) -- (4,4);
\node [text=\EdgeColor, fill=white] at (2-0.3,4+0.25) {$\phy_a^{x_0, x_1 + x_2} + ax_3$};
\draw [->-=0.4, fill=none] (0,0) -- (0,4);
\node [text=\EdgeColor, fill=white] at (0,2+0.4) {$\phy_a^{x_0 + x_1 + x_2, x_3}$};
\draw [->-=0.4] (4,0) -- (4,4);
\node [text=\EdgeColor, fill=white] at (4,2+0.4) {$ax_0 + \phy_a^{x_1 + x_2, x_3}$};
\begin{pgfonlayer}{bg}
 \draw [->-=0.4] (2,2) -- (6,2);
 \node [text=\EdgeColor] at (4-0.2,2-0.2) {$\phy_a^{x_0, x_1} + a(x_2 + x_3)$};
\end{pgfonlayer}
\draw [->-=0.4] (2,6) -- (6,6);
\node [above] [text=\EdgeColor] at (4,6) {$\phy_a^{x_0, x_1} + ax_2 + ax_3$};
\begin{pgfonlayer}{bg}
 \draw [->-=0.35] (2,2) -- (2,6);
 \node [text=\EdgeColor, fill=none] at (2+0.3,4-1.0) {$a(x_0 + x_1) + \phy_a^{x_2, x_3}$};
\end{pgfonlayer}
\draw [->-=0.35] (6,2) -- (6,6);
\node [text=\EdgeColor, fill=none] at (6+0.3,4-1.0) {$ax_0 + ax_1 + \phy_a^{x_2, x_3}$};
\begin{pgfonlayer}{bg}
 \draw [->-=0.35] (0,0) -- (2,2);
 \node [text=\EdgeColor, fill=white] at (1,1) {$\phy_a^{x_0 + x_1, x_2 + x_3}$};
\end{pgfonlayer}
\draw [->-=0.4] (0,4) -- (2,6);
\node [above] [text=\EdgeColor, fill=white] at (1,5) {$\phy_a^{x_0 + x_1, x_2} + ax_3$};
\draw [->-=0.35] (4,0) -- (6,2);
\node [text=\EdgeColor, fill=white] at (5,1) {$ax_0 + \phy_a^{x_1, x_2 + x_3}$};
\draw [->-=0.4] (4,4) -- (6,6);
\node [above] [text=\EdgeColor, fill=white] at (5,5) {$ax_0 + \phy_a^{x_1, x_2} + ax_3$};
\node [text=\FaceColor] at (2,2-\adjust) {$\phy_a^{x_0, x_1 + x_2, x_3}$};%
\node [text=\FaceColor] at (4,4+\adjust) {$\phy_a^{x_0, x_1} \op \phy_a^{x_2, x_3}$};%
\node [xslant=0.4, text=\FaceColor] at (3,1) {$\phy_a^{x_0, x_1, x_2+x_3}$};%
\node [xslant=0.4, text=\FaceColor] at (3,5) {$\phy_a^{x_0, x_1, x_2} + ax_3$};%
\node [yslant=0.6, text=\FaceColor] at (1,3) {$\phy_a^{x_0+x_1, x_2, x_3}$};%
\node [yslant=0.6, text=\FaceColor] at (5,3) {$ax_0 + \phy_a^{x_1, x_2, x_3}$};%
\end{tikzpicture}
\caption{The obstruction map $\obs{\phy^2} \colon \del I^3 \to \cat{T}$.}
\label{fig:Distrib3Cube}
\end{figure}
\end{ex}

We can now reinterpret higher distributivity as an inductive construction. %
The following lemma could be taken as an alternate to Definition~\ref{def:nDistrib}.

\begin{lem} \label{lem:nDistribInduc}
Le $n \geq 1$ and let $\cat{T}$ be a left linear $\Topp_*$-category. A (continuous) family of $n$-cubes in $\cat{T}$
\[
\phy^{n} = \left\{ \phy_a^{x_0, \ldots, x_n} \mid a, x_0, \ldots, x_n \in \cat{T} \right\}
\]
and is an $n$-distributor for $\cat{T}$ if and only if $\phy^{n-1}$ is an $(n-1)$-distributor for $\cat{T}$, and for all $a, x_0, x_1, \ldots, x_n \in \cat{T}$, the $n$-cube $\phy_a^{x_0, \ldots, x_n} \colon I^n \to \cat{T}$ extends the obstruction map from Definition~\ref{def:nObstruc}:
\[
\phy_a^{x_0, \ldots, x_n} \vert_{\del I^n} = \obs{\phy^{n-1}}_a^{x_0, \ldots, x_n} \colon \del I^n \to \cat{T}.
\]
In shorter notation: $\phy^n \vert_{\del I^n} = \obs{\phy^{n-1}}$.

Moreover, a sequence $(\phy^0, \phy^1, \phy^2, \ldots)$ is an $\infty$-distributor for $\cat{T}$ if and only if each $\phy^{n}$ is an $n$-distributor based on $\phy^{n-1}$.
\end{lem}

\section{Good distributors} \label{sec:GoodDistrib}

In this section, we describe additional conditions on a distributor that will be satisfied in a weakly bilinear mapping theory.

For every $a, x_0, \ldots, x_n \in \cat{T}$, the two maps being compared in the distributivity equation, namely $a(x_0 + \ldots + x_n)$ and $ax_0 + \ldots + ax_n$, factor through $(x_0, \ldots, x_n) \colon X \to A^{n+1}$, as illustrated in the diagram
\[
\xymatrix @C=3.5pc {
X \ar[dr]_-{x_0 + \ldots + x_n} \ar[r]^-{(x_0, \ldots, x_n)} & A^{n+1} \ar[d]^{+_A} \ar[r]^-{a^{n+1}} & B^{n+1} \ar[d]^{+_B} \\
& A \ar[r]_-{a} & B. \\
}
\]
Hence, for fixed $a \in \cat{T}$, there is a universal case to consider, with $(x_0, \ldots, x_n) = \id_{A^{n+1}}$. In other words, take $x_i = p_i \colon A^{n+1} \to A$, where the maps $p_0, \ldots, p_n \colon A^{n+1} \to A$ denote the projections onto the factors. %

Also, the equalities $a(x+0) = a(0+x) = ax$ hold strictly in $\cat{T}$. More generally, given inputs $x_0, \ldots, x_n$ with $x_i = 0$ for $i \neq k$, then the equality $a(x_0 + \ldots + x_n) = a x_k$ holds. In other words, for all maps $a \colon A \to B$ and $x \colon X \to A$, and index $0 \leq k \leq n$, consider the diagram in $\cat{T}$
\[
\xymatrix @C=3.5pc {
X \ar[d]_{x} \ar[r]^-{(0,\ldots,\overset{k^{\text{th}}}{x},\ldots,0)} & A^{n+1} \ar[d]^{+_A} \ar[r]^-{a^{n+1}} & B^{n+1} \ar[d]^{+_B} \\
A \ar[ur]^{i_k} \ar@{=}[r] & A \ar[r]_-{a} & B. \\
}
\]
The left half commutes, but the right square does \emph{not} commute. However, the large rectangle does commute, as both composites are equal to $ax \colon X \to B$. 
These observations lead to the following:

\begin{defn} \label{def:Good}
An $n$-distributor $\phy^n$ %
is \Def{good} if it satisfies the following properties.
\begin{enumerate}
\item \label{eq:Universality} (\emph{Universality}) For all $a, x_0, \ldots, x_n \in \cat{T}$, the equality 
\begin{equation} \label{eq:NaturalCube}
\phy_a^{x_0, \ldots, x_n} = \phy_a^{p_0, \ldots, p_n} \ot (x_0, \ldots, x_n)
\end{equation}
holds. Both sides are $n$-cubes in $\cat{T} \left( X, B \right)$.
\item \label{eq:Wedge} (\emph{Wedge condition}) For all maps $a \colon A \to B$ and $x \colon X \to A$ in $\cat{T}$, and index $0 \leq k \leq n$, the cube
\[
\phy_a^{0,\ldots,x,\ldots,0} \colon I^n \to \cat{T}(X,B)
\]
is the constant $n$-cube at $ax$.
\end{enumerate}
An $\infty$-distributor $\phy = (\phy^0, \phy^1, \phy^2, \ldots)$ is called \emph{good} if $\phy^n$ is good for all $n \geq 0$.
\end{defn}

\begin{rem}
If $\cat{T}$ happened to come from an ambient model category $\cat{C}$, then the restriction
\[
(i_0^*, \ldots, i_n^*) \colon \cat{T}(A^{n+1},B) \to \cat{T}(A,B)^{n+1} \cong \cat{C}(\bigvee_{i=0}^n A,B)
\]
corresponds to restriction along the inclusion of the wedge $\bigvee_{i=0}^n A \inj A^{n+1}$. The wedge condition says that no correction is needed when we restrict to the wedge.
\end{rem}

Note that if $\phy^{n}$ is good, then $\phy^{n-1}$ is automatically good as well. Also note that the $0$-distributor $\phy^0$ is good, trivially.

\begin{lem} \label{lem:UnivWedge}
Let $\phy^{n}$ be an $n$-distributor satisfying the universality condition. Then $\phy^n$ satisfies the wedge condition if and only if for every map $a \colon A \to B$ in $\cat{T}$, the composite
\[
\xymatrix @C=3.5pc {
I^n \ar[r]^-{\phy_a^{p_0, \ldots, p_n}} & \cat{T} (A^{n+1}, B) \ar[r]^-{(i_0^*, \ldots, i_n^*)} & \cat{T} (A, B)^{n+1}
}
\]
is the constant $n$-cube at $(a, \ldots, a)$.
\end{lem}

\begin{proof}
The wedge condition says that for all $x \colon X \to A$ and all index $0 \leq k \leq n$, the map
\[
\phy_a^{0,\ldots,x,\ldots,0} \colon I^n \to \cat{T} (X,B)
\]
is the constant $n$-cube $\cst_{ax}$ at $ax \in \cat{T}(X,B)$. Since $\phy^n$ satisfies universality, we have
\begin{align*}
\phy_a^{0,\ldots,x,\ldots,0} &= \phy_a^{p_0, \ldots, p_n} \ot (0,\ldots,x,\ldots,0) \\
&= \phy_a^{p_0, \ldots, p_n} \ot \left( i_k x \right) \\
&= \phy_a^{p_0, \ldots, p_n} \ot i_k \ot x. 
\end{align*}
This cube is the constant $n$-cube at $ax$ for all $x$ if and only if
\[
\phy_a^{p_0, \ldots, p_n} \ot i_k \colon I^n \to \cat{T}(A,B)
\]
is the constant $n$-cube at $a \in \cat{T}(A,B)$. %
\end{proof}

\begin{lem} \label{lem:UnivObstruc}
If an $(n-1)$-distributor $\phy^{n-1}$ satisfies the universality property \ref{def:Good} \eqref{eq:Universality}, then the obstruction to $n$-distributivity $\obs{\phy^{n-1}}$ satisfies the following ana\-lo\-gous property: For all $a, x_0, \ldots, x_n \in \cat{T}$, the equality 
\[
\obs{\phy^{n-1}}_a^{x_0, \ldots, x_n} = \obs{\phy^{n-1}}_a^{p_0, \ldots, p_n} \ot (x_0, \ldots, x_n)
\]
holds. Both sides are maps $\del I^n \to \cat{T} \left( X, B \right)$.
\end{lem}

\begin{proof}
Let us show that both sides agree when restricted to any face $C_{\si} \subseteq \del I^n$. The right-hand side is:
\begin{align*}
&\left( \obs{\phy^{n-1}}_a^{p_0, \ldots, p_n} \ot (x_0, \ldots, x_n) \right) \vert_{C_{\si}} \\
= &\left( \obs{\phy^{n-1}}_a^{p_0, \ldots, p_n} \right) \vert_{C_{\si}} \ot (x_0, \ldots, x_n) \\
= &\left( \phy^{n-1}[\si \vert_{J_0}]_a^{p_0, \ldots, p_n} \op \ldots \op \phy^{n-1}[\si \vert_{J_t}]_a^{p_0, \ldots, p_n} \right) \ot (x_0, \ldots, x_n) \\
= &\left( \phy^{n-1}[\si \vert_{J_0}]_a^{p_0, \ldots, p_n} \ot (x_0, \ldots, x_n) \right) \op \ldots \op \left( \phy^{n-1}[\si \vert_{J_t}]_a^{p_0, \ldots, p_n} \ot (x_0, \ldots, x_n) \right)
\end{align*}
using Lemma~\ref{lem:LeftLinCubes}. Here we used the notation of Definition~\ref{def:FaceFormula}, with the partition into intervals $\ord{n} = J_0 \sqcup \ldots \sqcup J_t$. Hence, it suffices to check the claim for each such interval $J$, itself partitioned into intervals $J = K_0 \sqcup \ldots \sqcup K_d$. Here the dimension $d$ satisfies $d < n$, since $C_{\si}$ is a face of the boundary $\del I^n$. By definition of the obstruction map $\obs{\phy^{n-1}}$, we have
\begin{align}
&\phy^{n-1}[\si \vert_{J}]_a^{p_0, \ldots, p_n} \ot (x_0, \ldots, x_n) \nonumber \\ 
= &\phy_a^{p_{K_0}, \ldots, p_{K_d}} \ot (x_0, \ldots, x_n) \nonumber \\
= &\phy_a^{\pi_0, \ldots, \pi_d} \ot \left( p_{K_0}, \ldots, p_{K_d} \right) \ot (x_0, \ldots, x_n) \nonumber
\end{align}
by universality of $\phy^{d}$, where we denoted the projection maps $\pi_i \colon A^{d+1} \to A$. Since the composite
\[
\xymatrix @C=3.5pc {
X \ar[r]^-{(x_0, \ldots, x_n)} & A^{n+1} \ar[r]^-{(p_{K_0}, \ldots, p_{K_d})} & A^{d+1} \\
}
\]
is equal to $(x_{K_0}, \ldots, x_{K_d}) \colon X \to A^{d+1}$, we obtain the further simplifications:
\begin{align*}
= &\phy_a^{\pi_0, \ldots, \pi_d} \ot \left( (p_{K_0}, \ldots, p_{K_d}) (x_0, \ldots, x_n) \right) \\
= &\phy_a^{\pi_0, \ldots, \pi_d} \ot \left( x_{K_0}, \ldots, x_{K_d} \right) \\
= &\phy_a^{x_{K_0}, \ldots, x_{K_d}} \quad \text{by universality of } \phy^{d} \\
= &\phy^{n-1}[\si \vert_{J}]_a^{x_0, \ldots, x_n}. \qedhere
\end{align*}
\end{proof}

\begin{cor} \label{cor:UnivSuffices}
Let $\phy^{n-1}$ be an $(n-1)$-distributor satisfying the universality condition \ref{def:Good} \eqref{eq:Universality}, and let
\[
\phy_a^{p_0, \ldots, p_n} \colon I^n \to \cat{T} \left( A^{n+1}, B \right)
\]
be an extension of the obstruction map $\obs{\phy^{n-1}}_a^{p_0, \ldots, p_n} \colon \del I^n \to \cat{T} \left( A^{n+1}, B \right)$, for each map $a \in \cat{T}$, depending continuously on $a$. Then the formula
\[
\phy_a^{x_0, \ldots, x_n} := \phy_a^{p_0, \ldots, p_n} \ot (x_0, \ldots, x_n)
\]
defines an $n$-distributor $\phy^{n}$ based on $\phy^{n-1}$. Note that $\phy^{n}$ also satisfies universality, by construction.
\end{cor}

\begin{proof}
The formula is well-defined and continuous in its inputs $a, x_0, \ldots x_n$. The restriction of $\phy^n$ to the boundary $\del I^n$ is:
\begin{align*}
\phy_a^{x_0, \ldots, x_n} \vert_{\del I^n} &= \left( \phy_a^{p_0, \ldots, p_n} \ot (x_0, \ldots, x_n) \right) \vert_{\del I^n} \\
&= \left( \phy_a^{p_0, \ldots, p_n} \vert_{\del I^n} \right) \ot (x_0, \ldots, x_n) \\
&= \left( \obs{\phy^{n-1}}_a^{p_0, \ldots, p_n} \right) \ot (x_0, \ldots, x_n) \quad \text{by assumption} \\
&= \obs{\phy^{n-1}}_a^{x_0, \ldots, x_n}
\end{align*}
by Lemma~\ref{lem:UnivObstruc}, using the fact that $\phy^{n-1}$ satisfies universality.
\end{proof}

\begin{lem} \label{lem:WedgeObstruc}
Let $\phy^{n-1}$ be an $(n-1)$-distributor for $\cat{T}$ satisfying the wedge condition, i.e., Definition~\ref{def:Good} \eqref{eq:Wedge}.
\begin{enumerate}
\item The obstruction to $n$-distributivity $\obs{\phy^{n-1}} \colon \del I^n \to \cat{T}$ satisfies the following ana\-lo\-gous condition: For all maps $a \colon A \to B$ and $x \colon X \to A$ in $\cat{T}$, and index $0 \leq k \leq n$, the map
\[
\obs{\phy^{n-1}}_a^{0,\ldots,x,\ldots,0} \colon \del I^n \to \cat{T}(X,B)
\]
is constant with value $ax$.
\item If moreover $\phy^{n-1}$ satisfies universality, i.e., Definition~\ref{def:Good} \eqref{eq:Universality}, then the property of $\obs{\phy^{n-1}}$ described in the previous part is equivalent to the following: For every map $a \colon A \to B$ in $\cat{T}$, the composite
\[
\xymatrix @C=4.8pc {
\del I^n \ar[r]^-{\obs{\phy^{n-1}}_a^{p_0, \ldots, p_n}} & \cat{T} (A^{n+1}, B) \ar[r]^-{(i_0^*, \ldots, i_n^*)} & \cat{T} (A, B)^{n+1}
}
\]
is constant with value $(a, \ldots, a)$.
\end{enumerate}
\end{lem}

\begin{proof}
(1) If suffices to show that for every face $C_{\si} \subset \del I^n$, the restriction $\obs{\phy^{n-1}}_a^{0,\ldots,x,\ldots,0} \vert_{C_{\si}} = \phy^{n-1}[\si]_a^{0,\ldots,x,\ldots,0}$ is constant with value $ax$. Consider the partition $\ord{n} = J_0 \sqcup \ldots \sqcup J_t$ as in Definition~\ref{def:FaceFormula}, with the chosen index $k$ satisfying $k \in J_l$ for some unique $0 \leq l \leq t$. For $i \neq l$, we have
\[
\phy^{n-1}[\si \vert_{J_i}]_a^{0,\ldots,x,\ldots,0} = \phy_a^{0,\ldots,0} = 0 \colon I^{\dim \si \vert_{J_i}} \to \cat{T}(X,B).
\]
For $i = l$, write the partition \[
J_l = K_0 \sqcup K_1 \sqcup \ldots \sqcup K_d
\]
as in Definition~\ref{def:FaceFormula}, with $k \in K_m$. Then we have
\[
x_{K_j} = \begin{cases}
x &\text{if } j = m \\
0 &\text{if } j \neq m \\
\end{cases}
\]
and therefore
\begin{align*}
\phy^{n-1}[\si \vert_{J_l}]_a^{0,\ldots,x,\ldots,0} &= \phy_a^{x_{K_0},\ldots,x_{K_m},\ldots,x_{K_d}} \\
&= \phy_a^{0,\ldots,x,\ldots,0} \\
&= ax \colon I^{\dim \si \vert_{J_l}} \to \cat{T}(X,B)
\end{align*}
since $\phy^{d}$ satisfies the wedge condition. Finally, we obtain:
\begin{align*}
\phy^{n-1}[\si]_a^{0,\ldots,x,\ldots,0} &= \phy^{n-1}[\si \vert_{J_0}]_a^{0,\ldots,x,\ldots,0} \op \cdots \op \phy^{n-1}[\si \vert_{J_l}]_a^{0,\ldots,x,\ldots,0} \op \cdots \op \phy^{n-1}[\si \vert_{J_t}]_a^{0,\ldots,x,\ldots,0} \\
&= 0 \op \cdots \op ax \op \cdots \op 0 \\
&= ax \colon I^{\dim \si} \to \cat{T}(X,B).
\end{align*}
(2) This uses the same argument as in Lemma~\ref{lem:UnivWedge}.
\end{proof}

\begin{lem} \label{lem:ProdCofib}
Let $i \colon X \to Y$ be a Serre cofibration between spaces, and let $L$ be a Serre cofibrant space. Then the map $i \x L \colon X \x L \to Y \x L$ is a Serre cofibration. 

The statement also holds with every instance of ``Serre'' replaced by ``mixed'', or every instance replaced by ``Hurewicz''.
\end{lem}

\begin{proof}
See \cite{MayP12}*{Theorems~17.1.1, 17.2.2, 17.4.2}.
\end{proof}

\begin{prop} \label{pr:ExistNDistrib}
Let $\cat{T}$ be a weakly bilinear mapping theory in which all mapping spaces $\cat{T}(A,B)$ are Serre cofibrant. Let $n \geq 1$, and let $\phy^{n-1}$ be a good $(n-1)$-distributor for $\cat{T}$. Then there exists a good $n$-distributor $\phy^{n}$ for $\cat{T}$ based on $\phy^{n-1}$. Moreover, %
such a $\phy^{n}$ 
is unique up to homotopy rel $\del I^n$.
\end{prop}

\begin{proof}
By Equation~\eqref{eq:NaturalCube}, it suffices to consider the universal case $x_i = p_i \colon A^{n+1} \to A$. Recall that the restriction $\phy_a^{p_0, \ldots, p_n} \vert_{\del I^n}$ must be the obstruction map $\obs{\phy^{n-1}} \colon \del I^n \to \cat{T} (A^{n+1}, B)$, which is determined by $\phy^{n-1}$. Since $\phy^{n-1}$ satisfies the wedge condition, the following square commutes:
\[
\xymatrix @C=\bigcol @R=1pc {
\del I^n \x \cat{T}(A,B) \ar@{^{(}->}[dd] \ar[r]^-{\obs{\phy^{n-1}}} & \cat{T} \left( A^{n+1}, B \right) \ar[dd]^{(i_0^*, \ldots, i_n^*)} \\
& \\
I^n \x \cat{T}(A,B) \ar@{-->}[uur]^{\phy_a^{p_0, \ldots, p_n}} \ar[r]_-{\De \circ \proj_2} & \cat{T} (A,B)^{n+1} \\
(t, a) \ar@{|->}[r] & (a, \ldots, a), \\
}
\]
using Lemma~\ref{lem:WedgeObstruc}. Since the space $\cat{T}(A,B)$ is Serre cofibrant by assumption, the downward map $\del I^n \x \cat{T}(A,B) \to I^n \x \cat{T}(A,B)$ is a Serre cofibration, by Lemma~\ref{lem:ProdCofib}. Since $(i_0^*, \ldots, i_n^*)$ is a trivial Serre fibration, there exists a dotted filler in the diagram. The top triangle guarantees that the collection of $n$-cubes
\[
\phy^n \dfn \left\{ \phy_a^{p_0, \ldots, p_n} \ot (x_0, \ldots, x_n) \mid a, x_0, \ldots, x_n \in \cat{T} \right\}
\]
defines an $n$-distributor for $\cat{T}$ which is based on $\phy^{n-1}$, using Corollary~\ref{cor:UnivSuffices}. By construction, $\phy^{n}$ satisfies universality. The bottom triangle guarantees that $\phy^{n}$ also satisfies the wedge condition, hence is good.

For uniqueness, let $\phy$ and $\phy'$ be two good extensions of $\obs{\phy^{n-1}}$ to $I^n$. These jointly define a map
\[
\xymatrix{
\left( I^n \x \cat{T}(A,B) \right) \cup_{\del I^n \x \cat{T}(A,B)} \left( I^n \x \cat{T}(A,B) \right) \ar[d]_{\cong} \ar[r]^-{\phy \cup \phy'} & \cat{T} \left( A^{n+1}, B \right) \\
(I^n \cup_{\del I^n} I^n) \x \cat{T}(A,B) \ar@{-}[d]_{\cong} \ar[ur] & \\
S^n \x \cat{T}(A,B). \ar[uur] & \\
}
\]
Again, there exists a filler in the diagram
\[
\xymatrix @C=\bigcol {
S^n \x \cat{T}(A,B) \ar@{^{(}->}[d] \ar[r]^-{\phy \cup \phy'} & \cat{T} (A^{n+1}, B) \ar[d]^{(i_0^*, \ldots, i_n^*)} \\
D^n \x \cat{T}(A,B) \ar@{-->}[ur] \ar[r]_-{\De \circ \proj_2} & \cat{T} (A, B)^{n+1}, \\
}
\]
which provides a homotopy rel $\del I^n$ between $\phy$ and $\phy'$.
\end{proof}

\begin{thm} \label{thm:InftyDistrib}
Let $\cat{T}$ be a weakly bilinear mapping theory in which every mapping space $\cat{T}(A,B)$ is Serre cofibrant. Then $\cat{T}$ admits a good $\infty$-distributor (as in Definition~\ref{def:Good}).
\end{thm}
Recall that the Serre cofibrant spaces are precisely the retracts of cell complexes, which include in particular CW complexes, in particular geometric realizations of simplicial sets.

\begin{proof}
Starting from the $0$-distributor $\phy^0$, inductively choose a good $n$-distributor $\phy^n$ based on $\phy^{n-1}$, for all $n \geq 1$, using Proposition~\ref{pr:ExistNDistrib}.
\end{proof}

\section{The Kristensen derivation} \label{sec:Kristensen}

In this section, we fix the prime $p=2$ and work with the mod $2$ Eilenberg--MacLane mapping theory $\cat{EM}$, as in Definition~\ref{def:EMMappingTh}. Recall that $K_n = \sh^n K_0$ denotes our preferred model for $\Si^n H\F_2$.

\subsection{The Kristensen derivation from $1$-distributivity} \label{sec:Kristensen1Dist}

Let $\phy^1$ be a good $1$-distributor for $\cat{EM}$; recall that $\phy^1$ consists of a collection of paths $\phy_a^{x,y}$ of the form illustrated here:
\[
\begin{tikzpicture}[scale=2]
\draw [->-=0.5] (0,0) -- (2,0);
\draw [fill=\VertexColor] (0,0) circle (\VertexSize);
\node [below left] at (0,0) {$a(x + y)$};
\draw [fill=\VertexColor] (2,0) circle (\VertexSize);
\node [below right] at (2,0) {$ax + ay$.};
\node [below, text=\EdgeColor] at (1,0) {$\phy_a^{x,y}$};
\end{tikzpicture}
\]
The following terminology and notation follows \cite{Baues06}*{\S 4.2}. 

\begin{defn}
The \Def{linearity tracks} for $\cat{EM}$ are the homotopy classes of the paths $\phy_a^{x,y}$ rel $\del I$, i.e., the tracks
\[
\Ga_a^{x,y} \dfn \{ \phy_a^{x,y} \}.
\]
\end{defn}

By Proposition~\ref{pr:ExistNDistrib}, $\Ga_a^{x,y}$ is well-defined, i.e., independent of the choice of a good $1$-distributor $\phy^1$. Now take an element of the Steenrod algebra $a \in \steen^m$ of degree $m$, represented by a map 
$a \colon K_0 \to K_{m}$.
Taking $x=y=1_{K_0}$, the linearity track $\Ga_a^{1,1}$ is a track in $\cat{EM}(K_0, K_{m})$ of the form
\[
\xymatrix{
0 = a0 = a(1 + 1) \ar@{=>}[r]^-{\Ga_a^{1,1}} & a 1 + a 1 = a + a = 0.
}
\]
Here, we used the fact that $K_n$ is an $\F_2$-vector space object, by Lemma~\ref{lem:GoodModelHA}. The track $\Ga_a^{1,1}$ is a well defined class
\[
\ka(a) \dfn \Ga_a^{1,1} \in \pi_1 \cat{EM}(K_0,K_m) = [H\F_2, \Si^{m-1} H\F_2] = \steen^{m-1}.
\]
This defines %
a function $\ka \colon \steen \to \steen$ of degree $-1$.

In what follows, we will use some of the linearity track equations \cite{Baues06}*{Theorem~4.2.5}.

\begin{lem} \label{lem:LinTrackEq}
The linearity $1$-tracks $\Ga_{a}^{x,y}$ satisfy the following equations.
\begin{enumerate}
\item Precomposition: $\Ga_a^{xd,yd} = \Ga_a^{x,y} d$.
\item Left linearity: $\Ga_{a+a'}^{x,y} = \Ga_a^{x,y} + \Ga_{a'}^{x,y}$.
\item Product rule: $\Ga_{ba}^{x,y} = \Ga_{b}^{ax,ay} \sq b \Ga_a^{x,y}$.
\end{enumerate}
\end{lem}

\begin{lem} \label{lem:Derivation}
The function $\ka \colon \steen \to \steen$ is a derivation.
\end{lem}

\begin{proof}
This is stated and sketched in \cite{Baues06}*{Lemma~4.5.5}. 

Linearity of $\ka$ follows from linearity of $\Ga_a^{x,y}$ in the input $a$. 
The derivation property $\ka(ba) = \ka(b) a + b \ka(a)$ follows from the precomposition equation and product rule in Lemma~\ref{lem:LinTrackEq}.
\end{proof}

\begin{prop} \label{pr:DerivationSq}
Applied to Steenrod squares, the function $\ka$ satisfies
\[
\ka(\Sq^m) = \Sq^{m-1}.
\]
In particular, $\ka$ agrees with the Kristensen derivation \cite{Kristensen63}*{\S 2}.
\end{prop}

\begin{proof}
This is proved in \cite{Baues06}*{Theorem 4.5.8}, using work in \cite{Kristensen63}.
\end{proof}

The existence of the derivation $\ka \colon \steen \to \steen$ is a non-trivial property of the Steenrod algebra, which can be checked explicitly using the Adem relations; c.f. \cite{Kristensen63}*{\S 2}.

\begin{rem}
It was pointed out to us by Fernando Muro that the Kristensen derivation $\ka \colon \steen \to \steen$ is also obtained from %
\cite{BauesM11}*{\S 1}. More precisely, consider the ring spectrum $R = \End_S(H\F_p, H\F_p)$, the endomorphism ring spectrum of $H\F_p$ as a module over the sphere spectrum $S$. The homotopy groups of $R$ are the Steenrod algebra with reversed grading:
\[
\pi_m R = [ S^m \sm H\F_p, H\F_p ] \cong [ H\F_p, S^{-m} \sm H\F_p ] = \steen^{-m}.
\]
The unit map $\eta \colon S \to R$ induces on homotopy the map
\[
\xymatrix{
\Z \cong \pi_0 S \ar[r]^-{\pi_0 \eta} & \pi_0 R \cong \F_p \\
}
\]
which sends $1$ to $1$. %
Taking the class $p \in \pi_0 S$ which lies in the kernel of $\pi_0 \eta$, the construction in \cite{BauesM11}*{\S 1} yields a function
\[
\te(p) \colon \pi_m R \to \pi_{m+1} R
\]
which %
is independent of the choice of nullhomotopy of $p 1_{H\F_p}$, %
because the indeterminacy lives in $\pi_{0+1} R = \steen^{-1} = 0$. The function $\te(p)$ sends a class $a \in \pi_m R$ to a certain self-track of zero $a p \Ra p a$. In the case $p=2$, this track coincides with the linearity track %
$\{ \phy_a^{1,1} \} \colon a2 = a(1+1) \Ra (1+1)a = 2a$.%
\end{rem}

\subsection{Linearity $2$-tracks}

\begin{defn}
Let $\cat{T}$ be a weakly bilinear mapping theory with Serre cofibrant mapping spaces. Let $\phy^2$ be a good $2$-distributor for $\cat{T}$. We call the homotopy class rel $\del I^2$ of $\phy_a^{x,y,z} \colon I^2 \to \cat{T}$, denoted $\{ \phy_a^{x,y,z} \}$, a \Def{linearity $2$-track}.
\end{defn}

Again by Proposition~\ref{pr:ExistNDistrib}, $\{ \phy_{a}^{x,y,z} \}$ is determined by the underlying $1$-distributor $\phy^1$. In this subsection, we work out a few equations satisfied by the linearity $2$-tracks, analogous to the equations satisfied by the linearity $1$-tracks $\{ \phy_a^{x,y} \}$ listed in Lemma~\ref{lem:LinTrackEq}.

\begin{lem}[Precomposition] \label{lem:2TrackEqPrecompo}
For all inputs $a,x,y,z,d \in \cat{T}$, %
the following equation of $2$-tracks holds: 
$\{ \phy_{a}^{xd,yd,zd} \} = \{ \phy_{a}^{x,y,z} \}d$.
\end{lem}

\begin{proof}
Since $\phy^2$ satisfies the universality condition, the equality holds even at the level of $2$-cubes, not merely $2$-tracks:
\begin{align*}
\phy_{a}^{xd,yd,zd} &= \phy_{a}^{p_0, p_1, p_2} \ot (xd, yd, zd) \\
&= \phy_{a}^{p_0, p_1, p_2} \ot (x, y, z) \ot d \\
&= \phy_{a}^{x, y, z} \ot d. \qedhere
\end{align*}
\end{proof}

Next, we work out the analogue of the left linearity equation of $1$-tracks $\Ga_{a+a'}^{x,y} = \Ga_{a}^{x,y} + \Ga_{a'}^{x,y}$. Note that $\{ \phy_{a+a'}^{x,y,z} \}$ and $\{ \phy_{a}^{x,y,z} \} + \{ \phy_{a'}^{x,y,z} \}$ are usually different, since they do not agree on the boundary $\del I^2$. To compare them, we need some correction $2$-tracks.

\begin{nota} \label{nota:LeftLinWitness}
Consider maps $a,a' \colon A \to B$ in $\cat{T}$. By the argument in Proposition~\ref{pr:ExistNDistrib}, there exists a path homotopy
\[
L_{a,a'}^{p_0,p_1} \colon \phy_{a+a'}^{p_0,p_1} \Ra \phy_{a}^{p_0,p_1} + \phy_{a'}^{p_0,p_1}
\]
between paths in $\cat{T}(A \x A, B)$ such that for all $k$, the restriction $i_k^* L_{a,a'}^{p_0,p_1}$ is the constant $2$-cube at $a+a' \in \cat{T}(A,B)$. Moreover, such a path homotopy is unique up to homotopy rel $\del I^2$, i.e., yields a well-defined globular $2$-track $\{ L_{a,a'}^{p_0,p_1} \}$. For arbitrary maps $x,y \colon X \to A$ in $\cat{T}$, define
\[
L_{a,a'}^{x,y} := L_{a,a'}^{p_0,p_1} \ot (x,y),
\]
which is a path homotopy in $\cat{T}(X,B)$ as illustrated here: 
\[
\begin{tikzcd}
a(x+y) + a'(x+y) \ar[rr, "\phy_a^{x,y} + \phy_{a'}^{x,y}" \EdgeColor] & \mbox{} & ax + ay + a'x + a'y \\
(a+a')(x+y) \ar[u, equal, line width=0.12ex, double distance=0.4ex] \ar[rr, "\phy_{a+a'}^{x,y}"' \EdgeColor] & \mbox{} \ar[u, Rightarrow, line width=0.14ex, double distance=0.5ex, "L_{a,a'}^{x,y}"' \FaceColor] & (a+a')x + (a+a')y \ar[u, equal, line width=0.12ex, double distance=0.4ex] \\
\end{tikzcd}
\]

\end{nota}

\begin{lem}[Left linearity] \label{lem:2TrackEqLeftLin}
The $2$-track illustrated in Figure~\ref{fig:Sum2Tracks} is equal to $\{ \phy_{a}^{x,y,z} \} + \{ \phy_{a'}^{x,y,z} \}$.

\begin{figure}[ht]
\hspace*{-0.6cm}
\begin{tikzpicture}[scale=1.8]
\draw [fill=\VertexColor] (0,0) circle (\VertexSize);
\draw [fill=\VertexColor] (0,2) circle (\VertexSize);
\draw [fill=\VertexColor] (2,0) circle (\VertexSize);
\draw [fill=\VertexColor] (2,2) circle (\VertexSize);
\draw [->-=0.5] (0,0) -- (0,2);
\draw [->-=0.5] (0,2) -- (2,2);
\draw [->-=0.5] (0,0) -- (2,0);
\draw [->-=0.5] (2,0) -- (2,2);
\draw [->-=0.5] (0,0) arc (270:90:1);
\draw [->-=0.5] (0,2) arc (180:0:1);
\draw [->-=0.5] (0,0) arc (180:360:1);
\draw [->-=0.5] (2,0) arc (-90:90:1);
\DoubleArrow (-0.3,1) -- (-0.7,1);
\DoubleArrow (0.5,2.3) -- (0.5,2.7);
\DoubleArrow (1,-0.3) -- (1,-0.7);
\DoubleArrow (2.3,1) -- (2.7,1);
\node [below left, align=right] at (0,0) {$(a+a')(x+y+z)$\\$= a(x+y+z) + a'(x+y+z)$};
\node [above left, align=right] at (0,2) {$(a+a')(x + y) + (a+a')z$\\$= a(x+y) + az + a'(x+y) + a'z$};
\node [below right, align=left] at (2,0) {$(a+a')x + (a+a')(y + z)$\\$= ax + a(y+z) + a'x + a'(y+z)$.};
\node [above right, align=left] at (2,2) {$(a+a')x + (a+a')y + (a+a')z$\\$= ax + ay + az + a'x + a'y + a'z$};
\node [above, text=\EdgeColor] at (1,0) {$\phy_{a+a'}^{x,y + z}$};
\node [below, text=\EdgeColor] at (1,2) {$\phy_{a+a'}^{x,y} + (a+a')z$};
\node [above right, text=\EdgeColor] at (0,1) {$\phy_{a+a'}^{x + y, z}$};
\node [below, text=\EdgeColor] at (2,1) {$(a+a')x + \phy_{a+a'}^{y, z}$};
\node [left, text=\EdgeColor] at (-1,1) {$\phy_{a}^{x + y, z} + \phy_{a'}^{x + y, z}$};
\node [above, text=\EdgeColor] at (1,3) {$\phy_{a}^{x,y} + \phy_{a'}^{x,y} + (a+a')z$};
\node [below, text=\EdgeColor] at (1,-1) {$\phy_{a}^{x,y + z} + \phy_{a'}^{x,y + z}$};
\node [below right, text=\EdgeColor] at (3,1) {$(a+a')x + \phy_{a}^{y,z} + \phy_{a'}^{y,z}$};
\node [text=\FaceColor] at (1,1) {$\phy_{a+a'}^{x, y, z}$};
\node [above, text=\FaceColor] at (-0.5,1) {$L_{a,a'}^{x+y, z}$};
\node [right, text=\FaceColor] at (0.5,2.5) {$L_{a,a'}^{x, y} + (a+a')z$};
\node [right, text=\FaceColor] at (1,-0.5) {$L_{a,a'}^{x, y+z}$};
\node [above, text=\FaceColor] at (2.5,1) {$(a+a')x + L_{a,a'}^{y,z}$};
\end{tikzpicture}
\caption{Relating the $2$-tracks $\{ \phy_{a+a'}^{x,y,z} \}$  and $\{ \phy_{a}^{x,y,z} \} + \{ \phy_{a'}^{x,y,z} \}$.}
\label{fig:Sum2Tracks}
\end{figure}
\end{lem}

\begin{proof}
The illustrated $2$-track and $\{ \phy_{a}^{x,y,z} \} + \{ \phy_{a'}^{x,y,z} \}$ have the same restriction to the boundary $\del I^2$, namely $\obs{\phy^1}_{a}^{x,y,z} + \obs{\phy^1}_{a'}^{x,y,z}$. 
When all inputs $x,y,z$ are zero except one $x_k$, then the illustrated $2$-track and $\{ \phy_{a}^{x,y,z} \} + \{ \phy_{a'}^{x,y,z} \}$ are both the constant $2$-track at $ax_k + a'x_k = (a+a')x_k \in \cat{T}(X,B)$. 
By universality, it suffices to prove the claim in the case $x_i = p_i \colon A^3 \to A$. 
The claimed equality of $2$-tracks then follows from the uniqueness argument in Proposition~\ref{pr:ExistNDistrib}.
\end{proof}

Next, we turn to the product rule. The linearity $1$-tracks $\Ga_a^{x,y}$ satisfy the equation
\[
\Ga_{ba}^{x,y} = \Ga_{b}^{ax,ay} \sq b \Ga_{a}^{x,y}.
\]
As before, let us exhibit a canonical globular $2$-track that witnesses this equality of $1$-tracks.

\begin{nota} \label{nota:ProductWitness}
Consider maps $a \colon A \to B$ and $b \colon B \to C$ in $\cat{T}$. Denote by 
\[
P_{b,a}^{x,y} \colon \phy_{b}^{ax,ay} \sq  b \phy_{a}^{x,y} \Ra \phy_{ba}^{x,y}
\]
the path homotopy in $\cat{T}(X,C)$ as illustrated here:
\[
\begin{tikzcd}
& \mbox{} & \\
ba(x+y) \ar[rr, bend left=40, "\phy_{ba}^{x,y}" \EdgeColor] \ar[r, "b \phy_{a}^{x,y}"' \EdgeColor] & b (ax + ay) \ar[r, "\phy_{b}^{ax,ay}"' \EdgeColor] \ar[u, Rightarrow, line width=0.15ex, double distance=0.55ex, "P_{b,a}^{x,y}"' \FaceColor] & bax + bay \\
\end{tikzcd}
\]
defined similarly to Notation~\ref{nota:LeftLinWitness}, yielding a well-defined globular $2$-track $\{ P_{b,a}^{x,y} \}$.%
\end{nota}

Distributors can be generalized by letting the inputs $x_i \in \cat{T}$ be continuous families instead of points, and then applying the distributor pointwise. 
We make this precise in the following notation.

\begin{nota}
Let $a \colon A \to B$ be a map in $\cat{T}$, and $v \colon V \to \cat{T}(X,A)$ and $w \colon W \to \cat{T}(X,A)$ maps of spaces. As in Definition~\ref{def:ProdCubes}, the external addition $v \op w \colon V \x W \to \cat{T}(X,A)$ is the composite
\[
\xymatrix{
V \x W \ar[r]^-{v \x w} & \cat{T}(X,A) \x \cat{T}(X,A) \ar[r]^-{+} & \cat{T}(X,A). \\ 
}
\]
The $1$-distributor $\phy_a^1$ %
applied to the inputs $v$ and $w$ is the map $\phy_a^{v,w} \colon V \x W \x I \to \cat{T}(X,B)$ defined as the composite
\[
\xymatrix @C=3pc {
V \x W \x I \ar[r]^-{v \x w \x \id} & \cat{T}(X,A) \x \cat{T}(X,A) \x I \ar[r]^-{\phy_a^1} & \cat{T}(X,B), \\
}
\]
viewed as a homotopy from $a(v \op w)$ to $av \op aw$. 
\end{nota}

\begin{lem}[Product rule] \label{lem:2TrackEqProduct}
The $2$-track illustrated in Figure~\ref{fig:Product2Track} is equal to $\{ \phy_{ba}^{x,y,z} \}$.
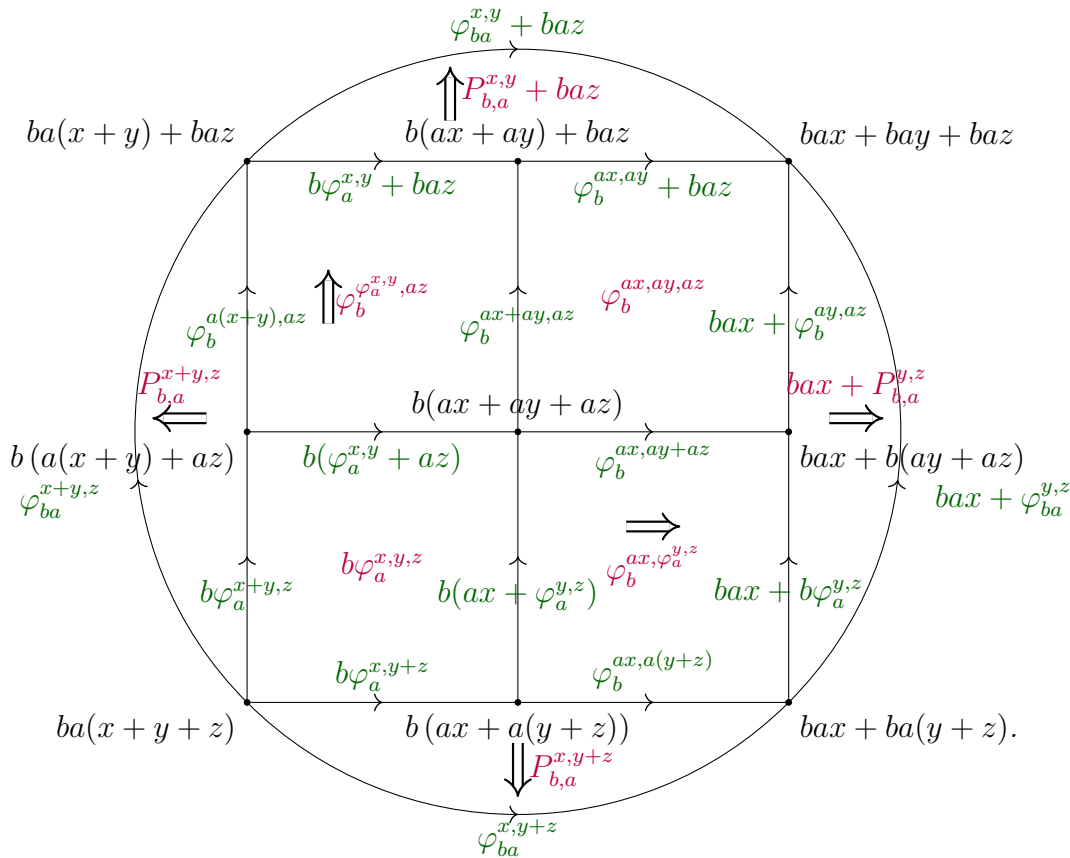
\begin{figure}[ht]
\begin{tikzpicture}[scale=1.8]
\draw [fill=\VertexColor] (0,0) circle (\VertexSize);
\draw [fill=\VertexColor] (0,2) circle (\VertexSize);
\draw [fill=\VertexColor] (0,4) circle (\VertexSize);
\draw [fill=\VertexColor] (2,0) circle (\VertexSize);
\draw [fill=\VertexColor] (2,2) circle (\VertexSize);
\draw [fill=\VertexColor] (2,4) circle (\VertexSize);
\draw [fill=\VertexColor] (4,0) circle (\VertexSize);
\draw [fill=\VertexColor] (4,2) circle (\VertexSize);
\draw [fill=\VertexColor] (4,4) circle (\VertexSize);
\draw [->-=0.5] (0,0) -- (2,0);
\draw [->-=0.5] (2,0) -- (4,0);
\draw [->-=0.5] (0,2) -- (2,2);
\draw [->-=0.5] (2,2) -- (4,2);
\draw [->-=0.5] (0,4) -- (2,4);
\draw [->-=0.5] (2,4) -- (4,4);
\draw [->-=0.54] (0,0) -- (0,2);
\draw [->-=0.54] (0,2) -- (0,4);
\draw [->-=0.54] (2,0) -- (2,2);
\draw [->-=0.54] (2,2) -- (2,4);
\draw [->-=0.54] (4,0) -- (4,2);
\draw [->-=0.54] (4,2) -- (4,4);
\draw [->-=0.42] (0,0) arc (225:135:{2*sqrt(2)});
\draw [->-=0.5] (0,4) arc (135:45:{2*sqrt(2)});
\draw [->-=0.5] (0,0) arc (225:315:{2*sqrt(2)});
\draw [->-=0.42] (4,0) arc (-45:45:{2*sqrt(2)});
\DoubleArrow (0.6,2.8) -- (0.6,3.2);
\DoubleArrow (2.8,1.3) -- (3.2,1.3);
\DoubleArrow (-0.3,2.1) -- (-0.7,2.1);
\DoubleArrow (1.5,4.3) -- (1.5,4.7);
\DoubleArrow (2,-0.3) -- (2,-0.7);
\DoubleArrow (4.3,2.1) -- (4.7,2.1);
\node [below left, align=right] at (0,0) {$ba(x+y+z)$};
\node [below left, align=right] at (0,2) {$b \left( a(x+y)+az \right)$};
\node [above left, align=right] at (0,4) {$ba(x+y)+ baz$};
\node [below, align=center] at (2,0) {$b \left( ax + a(y+z) \right)$};
\node [above, align=center] at (2,2) {$b (ax + ay + az)$};
\node [above, align=center] at (2,4) {$b (ax+ay)+ baz$};
\node [below right, align=left] at (4,0) {$bax + ba(y+z)$.};
\node [below right, align=left] at (4,2) {$bax + b(ay+az)$};
\node [above right, align=left] at (4,4) {$bax + bay + baz$};
\node [below, text=\EdgeColor] at (0,1) {$b \phy_{a}^{x + y, z}$};
\node [below, text=\EdgeColor] at (0,3) {$\phy_{b}^{a(x+y), az}$};
\node [below, text=\EdgeColor] at (2,1) {$b (ax + \phy_a^{y,z})$};
\node [below, text=\EdgeColor] at (2,3) {$\phy_b^{ax+ay,az}$};
\node [below, text=\EdgeColor] at (4,1) {$bax + b \phy_a^{y,z}$};
\node [below, text=\EdgeColor] at (4,3) {$bax + \phy_b^{ay,az}$};
\node [above, text=\EdgeColor] at (1,0) {$b \phy_{a}^{x,y+z}$};
\node [above, text=\EdgeColor] at (3,0) {$\phy_{b}^{ax,a(y+z)}$};
\node [below, text=\EdgeColor] at (1,2) {$b (\phy_{a}^{x,y} + az)$};
\node [below, text=\EdgeColor] at (3,2) {$\phy_{b}^{ax,ay+az}$};
\node [below, text=\EdgeColor] at (1,4) {$b \phy_{a}^{x,y} + baz$};
\node [below, text=\EdgeColor] at (3,4) {$\phy_{b}^{ax,ay} + baz$};
\node [left, text=\EdgeColor] at (-1,1.5) {$\phy_{ba}^{x+y, z}$};
\node [text=\EdgeColor] at (2,5) {$\phy_{ba}^{x,y} + baz$};
\node [text=\EdgeColor] at (2,-1) {$\phy_{ba}^{x,y+z}$};
\node [right, text=\EdgeColor] at (5,1.5) {$bax + \phy_{ba}^{y,z}$};
\node [text=\FaceColor] at (1,1) {$b \phy_{a}^{x,y,z}$};
\node [text=\FaceColor] at (1,3) {$\phy_{b}^{\phy_{a}^{x,y}, az}$};
\node [text=\FaceColor] at (3,1) {$\phy_{b}^{ax, \phy_{a}^{y,z}}$};
\node [text=\FaceColor] at (3,3) {$\phy_{b}^{ax,ay,az}$};
\node [above, text=\FaceColor] at (-0.5, 2.1) {$P_{b,a}^{x+y, z}$};
\node [right, text=\FaceColor] at (1.5, 4.5) {$P_{b,a}^{x,y} + baz$};
\node [right, text=\FaceColor] at (2, -0.5) {$P_{b,a}^{x, y+z}$};
\node [above, text=\FaceColor] at (4.5, 2.1) {$bax + P_{b,a}^{y,z}$};
\end{tikzpicture}
\caption{Relating the $2$-track $\{ \phy_{ba}^{x,y,z} \}$ to $\{ \phy_{a}^{x,y,z} \}$ and $\{ \phy_{b}^{ax,ay,az} \}$.}
\label{fig:Product2Track}
\end{figure}
\end{lem}

\begin{proof}
This is similar to the proof of Lemma~\ref{lem:2TrackEqLeftLin}.
\end{proof}

\begin{lem} \label{lem:FlatBox}
Consider a map $a \colon A \to B$, a $2$-cube $u \colon I^2 \to \cat{T}(X,A)$, a point $y \in \cat{T}(X,A)$. Then the $2$-track illustrated in Figure~\ref{fig:FlatBox} is equal to $\{ au + ay \}$.
\begin{figure}[ht]
\begin{tikzpicture}[scale=1.4]
\draw [->-=0.5] (0,0) -- (6,0);
\draw [->-=0.5] (2,2) -- (4,2);
\draw [->-=0.5] (2,4) -- (4,4);
\draw [->-=0.5] (0,6) -- (6,6);
\draw [->-=0.5] (0,0) -- (0,6);
\draw [->-=0.5] (2,2) -- (2,4);
\draw [->-=0.5] (4,2) -- (4,4);
\draw [->-=0.5] (6,0) -- (6,6);
\draw [->-=0.5] (2,2) -- (0,0);
\draw [->-=0.5] (4,2) -- (6,0);
\draw [->-=0.5] (2,4) -- (0,6);
\draw [->-=0.5] (4,4) -- (6,6);
\draw [fill=\VertexColor] (2,2) circle (\VertexSize);
\node [below left] at (2,2) {$a(u_{00} + y)$};
\draw [fill=\VertexColor] (4,2) circle (\VertexSize);
\node [below right] at (4,2) {$a(u_{10} + y)$};
\draw [fill=\VertexColor] (2,4) circle (\VertexSize);
\node [above left] at (2,4) {$a(u_{01} + y)$};
\draw [fill=\VertexColor] (4,4) circle (\VertexSize);
\node [above right] at (4,4) {$a(u_{11} + y)$};
\draw [fill=\VertexColor] (0,0) circle (\VertexSize);
\node [below left] at (0,0) {$au_{00} + ay$};
\draw [fill=\VertexColor] (6,0) circle (\VertexSize);
\node [below right] at (6,0) {$au_{10} + ay$};
\draw [fill=\VertexColor] (0,6) circle (\VertexSize);
\node [above left] at (0,6) {$au_{01} + ay$};
\draw [fill=\VertexColor] (6,6) circle (\VertexSize);
\node [above right] at (6,6) {$au_{11} + ay$};
\DoubleArrow (2.6,1.2) -- (2.6,0.8);
\DoubleArrow (0.9,2.6) -- (0.5,2.6);
\DoubleArrow (5.1,2.6) -- (5.5,2.6);
\DoubleArrow (2.6,4.8) -- (2.6,5.2);
\node [below, text=\EdgeColor] at (3,2) {$a(u_{I0} + y)$}; 
\node [above, text=\EdgeColor] at (3,4) {$a(u_{I1} + y)$};
\node [left, text=\EdgeColor] at (2,3) {$a(u_{0I} + y)$};
\node [right, text=\EdgeColor] at (4,3) {$a(u_{1I} + y)$};
\node [below, text=\EdgeColor] at (3,0) {$au_{I0} + ay$}; 
\node [above, text=\EdgeColor] at (3,6) {$au_{I1} + ay$};
\node [left, text=\EdgeColor] at (0,3) {$au_{0I} + ay$};
\node [right, text=\EdgeColor] at (6,3) {$au_{1I} + ay$};
\node [below right, text=\EdgeColor] at (1,1) {$\phy_{a}^{u_{00},y}$}; 
\node [below left, text=\EdgeColor] at (5,1) {$\phy_{a}^{u_{10},y}$}; 
\node [above right, text=\EdgeColor] at (1,5) {$\phy_{a}^{u_{01},y}$}; 
\node [above left, text=\EdgeColor] at (5,5) {$\phy_{a}^{u_{11},y}$}; 
\node [text=\FaceColor] at (3,3) {$a(u+y)$};
\node [text=\FaceColor] at (3,1) {$\phy_{a}^{u_{I0},y}$};
\node [below left, text=\FaceColor] at (1,3) {$\phy_{a}^{u_{0I},y}$};
\node [below right, text=\FaceColor] at (5,3) {$\phy_{a}^{u_{1I},y}$};
\node [text=\FaceColor] at (3,5) {$\phy_{a}^{u_{I1},y}$};
\end{tikzpicture}
\caption{Comparing the $2$-tracks $\{ a(u+y) \}$ and $\{ au + ay \}$.}
\label{fig:FlatBox}
\end{figure}
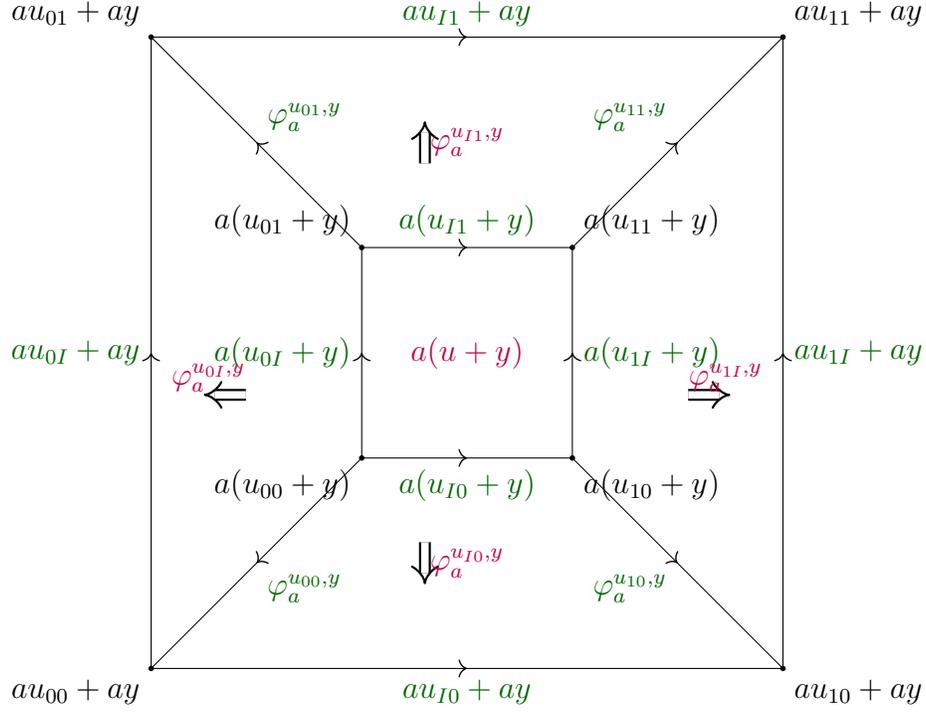
\end{lem}

\begin{proof}
The $3$-cube $\phy_a^{u,y} \colon I^3 \to \cat{T}(X,A)$ provides a homotopy rel $\del I^2$ between the illustrated $2$-track and $\{ au + ay \}$. 
\end{proof}

\subsection{Two-dimensional analogue of the derivation}

Let $\phy^2$ be a good $2$-distributor for $\cat{EM}$. 
As in the section~\ref{sec:Kristensen1Dist}, start with an element of the Steenrod algebra $a \in \steen^m$, represented by a map $a \colon K_0 \to K_{m}$. The $2$-cube $\phy_a^{1,1,1} \colon I^2 \to \cat{EM}$ restricts to the boundary $\del I^2$ as illustrated in Figure~\ref{fig:2Distrib111}.

\begin{figure}[ht]
\begin{tikzpicture}[scale=1.5]
\draw [->-=0.5] (0,0) -- (0,2);
\draw [->-=0.5] (0,2) -- (2,2);
\draw [->-=0.5] (0,0) -- (2,0);
\draw [->-=0.5] (2,0) -- (2,2);
\draw [fill=\VertexColor] (0,0) circle (\VertexSize);
\node [below left] at (0,0) {$a = a(1 + 1 + 1)$};
\draw [fill=\VertexColor] (0,2) circle (\VertexSize);
\node [above left] at (0,2) {$a = a(1 + 1) + a 1$};
\draw [fill=\VertexColor] (2,0) circle (\VertexSize);
\node [below right] at (2,0) {$a1 + a(1 + 1) = a$.};
\draw [fill=\VertexColor] (2,2) circle (\VertexSize);
\node [above right] at (2,2) {$a1 + a1 + a1 = a$};
\node [below, text=\EdgeColor] at (1,0) {$\phy_a^{1,0} = a$}; %
\node [above, text=\EdgeColor] at (1,2) {$\phy_a^{1,1} + a$};
\node [left, text=\EdgeColor] at (0,1) {$a = \phy_a^{0,1}$};%
\node [right, text=\EdgeColor] at (2,1) {$a + \phy_a^{1,1}$};
\node [text=\FaceColor] at (1,1) {$\phy_a^{1,1,1}$};
\end{tikzpicture}
\caption{The $2$-cube $\phy_a^{1,1,1} \colon I^2 \to \cat{EM}$.}
\label{fig:2Distrib111}
\end{figure}

Note that the equations $\phy_a^{1,0} = \cst_a = \phy_a^{0,1}$ are instances of the wedge condition satisfied by the $1$-distributor $\phy^1$. Subtracting $a$ pointwise yields the $2$-cube $\phy_a^{1,1,1} - a \colon I^2 \to \cat{EM}$ illustrated in Figure~\ref{fig:Corrected2Distrib}. The top right part uses the canonical path homotopy $\ep \colon \ga^{\inv} \sq \ga \Ra \cst_{\ga(0)}$ to the constant path at $\ga(0)$.

\begin{figure}[ht]
\begin{tikzpicture}[scale=1.6]
\draw [->-=0.5] (0,0) -- (0,2);
\draw [->-=0.6] (0,2) -- (2,2);
\draw [->-=0.5] (0,0) -- (2,0);
\draw [->-=0.6] (2,0) -- (2,2);
\draw (0,2) -- (3,3) -- (2,0);
\draw [fill=\VertexColor] (0,0) circle (\VertexSize);
\node [below left] at (0,0) {$0$};
\draw [fill=\VertexColor] (0,2) circle (\VertexSize);
\node [above left] at (0,2) {$0$};
\draw [fill=\VertexColor] (2,0) circle (\VertexSize);
\node [below right] at (2,0) {$0$.};
\draw [fill=\VertexColor] (2,2) circle (\VertexSize);
\node [above right] at (2,2) {$0$};
\draw [fill=\VertexColor] (3,3) circle (\VertexSize);
\node [right] at (3,3) {$0$};
\node [below, text=\EdgeColor] at (1,0) {$0$}; %
\node [above, text=\EdgeColor] at (1.3,2) {$\phy_a^{1,1}$};
\node [left, text=\EdgeColor] at (0,1) {$0$};
\node [right, text=\EdgeColor] at (2,1.3) {$\phy_a^{1,1}$};
\node [above, text=\EdgeColor] at (1.5,2.5) {$0$};
\node [right, text=\EdgeColor] at (2.5,1.5) {$0$};
\node [text=\FaceColor] at (1,1) {$\phy_a^{1,1,1} - a$};
\DoubleArrow (2.35,2.35) -- (2.65,2.65);
\node [above left, text=\FaceColor] at (2.5,2.5) {$\ep$};
\end{tikzpicture}
\caption{The $2$-cube $\phy_a^{1,1,1} - a \colon I^2 \to \cat{EM}$ and a correction term.}
\label{fig:Corrected2Distrib}
\end{figure}

Taking the homotopy class rel $\del I^2$, this construction yields a well-defined class 
\[
\la(a) \in \pi_2 \cat{EM}(K_0,K_m) \cong \steen^{m-2},
\]
and thus a function $\la \colon \steen \to \steen$ of degree $-2$.

\begin{lem}
The function $\la \colon \steen \to \steen$ is linear, i.e., preserves addition.
\end{lem}

\begin{proof}
Let $a,a' \in \steen^m$. Applying Lemma~\ref{lem:2TrackEqLeftLin} to the $2$-track $\{ \phy_{a+a'}^{1,1,1} \}$ and using the fact that $\{ L_{a,a'}^{0,1} \}$ and $\{ L_{a,a'}^{1,0} \}$ are both the constant $2$-track at $a+a'$, we obtain $\la(a+a') = \la(a) + \la(a')$.
\end{proof}

\begin{prop} \label{pr:LaDerivation}
The function $\la \colon \steen \to \steen$ is a derivation.
\end{prop}

\begin{proof}
Let $a \colon A \to B$ and $b \colon B \to C$ be maps in $\cat{EM}$. Applying Lemma~\ref{lem:2TrackEqProduct} to the case $x=y=z=1_A$, the $2$-track illustrated in Figure~\ref{fig:2TrackDerivation} is equal to $\{ \phy_{ba}^{1,1,1} \}$.

\begin{figure}[ht]
\begin{tikzpicture}[scale=1.6]
\draw [fill=\VertexColor] (0,0) circle (\VertexSize);
\draw [fill=\VertexColor] (0,2) circle (\VertexSize);
\draw [fill=\VertexColor] (0,4) circle (\VertexSize);
\draw [fill=\VertexColor] (2,0) circle (\VertexSize);
\draw [fill=\VertexColor] (2,2) circle (\VertexSize);
\draw [fill=\VertexColor] (2,4) circle (\VertexSize);
\draw [fill=\VertexColor] (4,0) circle (\VertexSize);
\draw [fill=\VertexColor] (4,2) circle (\VertexSize);
\draw [fill=\VertexColor] (4,4) circle (\VertexSize);
\draw [->-=0.5] (0,0) -- (2,0);
\draw [->-=0.5] (2,0) -- (4,0);
\draw [->-=0.5] (0,2) -- (2,2);
\draw [->-=0.5] (2,2) -- (4,2);
\draw [->-=0.5] (0,4) -- (2,4);
\draw [->-=0.5] (2,4) -- (4,4);
\draw [->-=0.54] (0,0) -- (0,2);
\draw [->-=0.54] (0,2) -- (0,4);
\draw [->-=0.54] (2,0) -- (2,2);
\draw [->-=0.54] (2,2) -- (2,4);
\draw [->-=0.54] (4,0) -- (4,2);
\draw [->-=0.54] (4,2) -- (4,4);
\draw [->-=0.5] (0,4) arc (135:45:{2*sqrt(2)});
\draw [->-=0.42] (4,0) arc (-45:45:{2*sqrt(2)});
\DoubleArrow (0.6,2.8) -- (0.6,3.2);
\DoubleArrow (2.8,1.3) -- (3.2,1.3);
\DoubleArrow (1.5,4.3) -- (1.5,4.7);
\DoubleArrow (4.3,2.1) -- (4.7,2.1);
\node [above, align=center] at (2,4) {$ba$};
\node [below right, align=left] at (4,2) {$ba$};
\node [above right, align=left] at (4,4) {$ba$};
\node [left, text=\EdgeColor] at (0,1) {$ba$};
\node [left, text=\EdgeColor] at (0,3) {$ba$};
\node [below, text=\EdgeColor] at (2,1) {$b (a + \phy_a^{1,1})$};
\node [below, text=\EdgeColor] at (2,3) {$ba$};
\node [below, text=\EdgeColor] at (4,1) {$ba + b \phy_a^{1,1}$};
\node [below, text=\EdgeColor] at (4,3) {$ba + \phy_b^{1,1}a$};
\node [below, text=\EdgeColor] at (1,0) {$ba$};
\node [below, text=\EdgeColor] at (3,0) {$ba$};
\node [below, text=\EdgeColor] at (1,2) {$b (\phy_{a}^{1,1} + a)$};
\node [below, text=\EdgeColor] at (3,2) {$ba$};
\node [below, text=\EdgeColor] at (1,4) {$b \phy_{a}^{1,1} + ba$};
\node [below, text=\EdgeColor] at (3,4) {$\phy_{b}^{1,1}a + ba$};
\node [text=\EdgeColor] at (2,5) {$\phy_{ba}^{1,1} + ba$};
\node [right, text=\EdgeColor] at (5,1.5) {$ba + \phy_{ba}^{1,1}$};
\node [text=\FaceColor] at (1,1) {$b \phy_{a}^{1,1,1}$};
\node [text=\FaceColor] at (1,3) {$\phy_{b}^{\phy_{a}^{1,1}, a}$};
\node [text=\FaceColor] at (3,1) {$\phy_{b}^{a, \phy_{a}^{1,1}}$};
\node [text=\FaceColor] at (3,3) {$\phy_{b}^{1,1,1}a$};
\node [right, text=\FaceColor] at (1.5, 4.5) {$P_{b,a}^{1,1} + ba$};
\node [above, text=\FaceColor] at (4.5, 2.1) {$ba + P_{b,a}^{1,1}$};
\end{tikzpicture}
\caption{The $2$-track $\{ \phy_{ba}^{1,1,1} \}$.}
\label{fig:2TrackDerivation}
\end{figure}

We used the fact that both $P_{b,a}^{1,0}$ and $P_{b,a}^{0,1}$ are the constant $2$-track at $ba \in \cat{EM}(A,C)$. We also used the precomposition equation from Lemma~\ref{lem:2TrackEqPrecompo}. Denote the $2$-track $u_a \dfn \{ \phy_a^{1,1,1} - a \}$. Applying Lemma~\ref{lem:FlatBox} to $b \phy_a^{1,1,1} = b(u_a + a)$, subtracting $ba$ pointwise everywhere, and applying the correction $2$-track $\ep$ in the upper right part of the diagram, we deduce that the $2$-track $\la(ba)$ is given as in Figure~\ref{fig:LambdaProduct}.

\begin{figure}[ht]
\begin{tikzpicture}[scale=1.5]
\draw [fill=\VertexColor] (0,0) circle (\VertexSize);
\draw [fill=\VertexColor] (0,4) circle (\VertexSize);
\draw [fill=\VertexColor] (2,2) circle (\VertexSize);
\draw [fill=\VertexColor] (2,4) circle (\VertexSize);
\draw [fill=\VertexColor] (4,0) circle (\VertexSize);
\draw [fill=\VertexColor] (4,2) circle (\VertexSize);
\draw [fill=\VertexColor] (4,4) circle (\VertexSize);
\draw [fill=\VertexColor] (5,5) circle (\VertexSize);
\draw [->-=0.5] (0,0) -- (4,0);
\draw [->-=0.5] (2,2) -- (4,2);
\draw [->-=0.5] (0,4) -- (2,4);
\draw [->-=0.5] (2,4) -- (4,4);
\draw [->-=0.50] (0,0) -- (0,4);
\draw [->-=0.50] (2,2) -- (2,4);
\draw [->-=0.50] (4,0) -- (4,2);
\draw [->-=0.50] (4,2) -- (4,4);
\draw [->-=0.5] (0,4) -- (2,2);
\draw [->-=0.5] (4,0) -- (2,2);
\draw [->-=0.5] (0,4) -- (5,5);
\draw [->-=0.5] (4,0) -- (5,5);
\DoubleArrow (3.65,1.65) -- (3.35,1.35);
\DoubleArrow (1.65,3.65) -- (1.35,3.35);
\DoubleArrow (4.35,4.35) -- (4.65,4.65);
\node [above, align=center] at (2,4) {$0$};
\node [below right, align=left] at (4,2) {$0$};
\node [above right, align=left] at (4,4) {$0$};
\node [left, text=\EdgeColor] at (0,2) {$0$};
\node [below left, text=\EdgeColor] at (3,1) {$b \phy_a^{1,1}$};
\node [left, text=\EdgeColor] at (2,3) {$0$};
\node [left, text=\EdgeColor] at (4,1) {$b \phy_a^{1,1}$};
\node [left, text=\EdgeColor] at (4,3) {$\phy_b^{1,1} a$};
\node [below, text=\EdgeColor] at (2,0) {$0$};
\node [below left, text=\EdgeColor] at (1,3) {$b \phy_{a}^{1,1}$};
\node [below, text=\EdgeColor] at (3,2) {$0$};
\node [below, text=\EdgeColor] at (1,4) {$b \phy_{a}^{1,1}$};
\node [below, text=\EdgeColor] at (3,4) {$\phy_{b}^{1,1}a$};
\node [above, text=\EdgeColor] at (2.5,4.5) {$0$};
\node [right, text=\EdgeColor] at (4.5,2.5) {$0$};
\node [text=\FaceColor] at (1,1) {$b u_a$};
\node [text=\FaceColor] at (3,3) {$u_b a$};
\node [above left, text=\FaceColor] at (1.5, 3.5) {$\eta$};
\node [above left, text=\FaceColor] at (3.5, 1.5) {$\eta$};
\node [above left, text=\FaceColor] at (4.5, 4.5) {$\ep$};
\end{tikzpicture}
\caption{The $2$-track $\la(ba)$.}
\label{fig:LambdaProduct}
\end{figure}

There, we denote by $\eta \colon \cst_{\ga(1)} \sq \ga \Ra \ga$ and $\eta \colon \ga \sq \cst_{\ga(0)} \Ra \ga$ the canonical path homotopies, whose $2$-tracks are well-defined. 
Straightforward manipulations of $2$-tracks yield the claimed equality $\la(ba) = \la(b) a + b \la(a)$.
\end{proof}

When working with mod $2$ coefficients, a composite of derivations is still a derivation. 
We leave the following question to the reader.

\begin{ques}
Is the derivation $\la \colon \steen \to \steen$ given by the composite $\la = \ka^2$? 
\end{ques}

\section{Homotopy invariance} \label{sec:HomotInvar}

In this section, we study to what extent an $n$-distributor is a homotopy invariant structure, and prove some homotopy transfer results. 
Since our construction of distributors relied on model dependent features (fibrant, cofibrant monoid objects in a simplicial model category), homotopy invariance provides 
some flexibility in the choice of model.
Unlike in Sections~\ref{sec:GoodDistrib} and \ref{sec:Kristensen}, goodness of distributors will play no role in this section. Also, finite products in $\cat{T}$, which were crucial to the construction of distributors, are not used here. Hence, instead of left linear mapping theories, we work with left linear $\Topp_*$-categories.

\subsection{Pulling back distributors}

\begin{lem} \label{lem:AddCubes}
Let $\cat{S}$ and $\cat{T}$ be $\Topp$-categories in which all mapping spaces are topological  %
monoids. Let $F \colon \cat{S} \to \cat{T}$ be a $\Topp$-functor such that for all objects $A,B$ of $\cat{S}$, the induced map
\[
F \colon \cat{S}(A,B) \to \cat{T}(FA,FB)
\]
is a %
monoid homomorphism. Then $F$ preserves external addition of cubes, i.e., given $a \colon I^m \to \cat{S}(A,B)$ and $b \colon I^n \to \cat{S}(A,B)$, the equality
\[
F(a \op b) = F(a) \op F(b)
\]
holds. Both sides are $(m+n)$-cubes in $\cat{T}(FA,FB)$.
\end{lem}

\begin{lem} \label{lem:PreserveObstruc}
Let $F \colon \cat{S} \to \cat{T}$ be a morphism of left linear $\Topp_*$-categories, and $\phy^{n-1}$ an $(n-1)$-distributor for $\cat{S}$. Then every proper subcube $C_{\si} \subset I^n$, we have
\[
F \left( \phy^{n-1}[\si] \right) = \left( F \phy^{n-1} \right) [\si]
\]
as maps $C_{\si} \to \cat{T}$. 
Consequently the obstruction map satisfies
\[
F \left( \obs{\phy^{n-1}} \right) = \obs{F \phy^{n-1}}
\]
as maps $\del I^n \to \cat{T}$.
\end{lem}

\begin{proof}
Consider the partition $\{ 0, 1, \ldots, n \} = J_0 \sqcup J_1 \sqcup \ldots \sqcup J_t$ as in Definition~\ref{def:FaceFormula}. Then we have
\begin{flalign*}
&& F \left( \phy^{n-1}[\si] \right) &= F \left( \phy^{n-1}[\si \vert_{J_0}] \op \cdots \op \phy^{n-1}[\si \vert_{J_t}] \right) && \\
&& &= F \left( \phy^{n-1}[\si \vert_{J_0}] \right) \op \cdots \op F \left( \phy^{n-1}[\si \vert_{J_t}] \right) && \text{by Lemma~\ref{lem:AddCubes}} \\
&& &= \left( F \phy^{n-1} \right) [\si \vert_{J_0}] \op \cdots \op \left( F \phy^{n-1} \right) [\si \vert_{J_t}] && \text{using } F(x+x') = Fx + Fx' \\
&& &= \left( F \phy^{n-1} \right) [\si]. && \qedhere
\end{flalign*}
\end{proof}

\begin{lem} \label{lem:LiftWeakEq}
Let $i \colon A \to X$ be a mixed cofibration and $q \colon Y \ral{\sim} Z$ a weak equivalence. Given a commutative square as in the diagram
\[
\xymatrix{
A \ar[d]_{i} \ar[r]^-{f} & Y \ar[d]^{q} \\
X \ar@{-->}[ur]^-{\tild{g}} \ar@{}@<2.0ex>[r]^(0.75){\Uparrow} \ar[r]^{g} & Z, \\
}
\]
there exists a map $\tild{g} \colon X \to Y$ making the upper triangle commute strictly, and the lower triangle commute up to homotopy rel $A$.
\end{lem}

\begin{proof}
Recall that the mapping path space $P(q) = Y \x_Z Z^I$ provides a (functorial) factorization of $q \colon Y \to Z$ into a strong deformation retract $c \colon Y \to P(q)$ followed by a Hurewicz fibration $p \colon P(q) \to Z$. Denote the retraction map by $r \colon P(f) \to Y$. In our case, $p \colon P(q) \to Z$ is also a weak equivalence, since $q$ was. In the diagram
\[
\xymatrix{
A \ar[dd]_{i} \ar[dr]^-{cf} \ar[rr]^-{f} & & Y \ar[dl]_-{c} \ar[dd]^{q}_{\sim} \\
& P(q) \ar@{->>}[dr]^-{p}_{\sim} \ar@/_0.6pc/[ur]_{r} & \\
X \ar@{-->}[ur]^-{g'} \ar[rr]^{g} & & Z, \\
}
\]
there is a map $g' \colon X \to P(q)$ making the two adjacent triangles commute. Indeed, $i \colon A \inj X$ is a mixed cofibration, whereas $p \colon P(q) \surj Z$ is a mixed trivial fibration (i.e., a Hurewicz fibration which is also a weak equivalence). Take $\tild{g} = rg' \colon X \to Y$. This map satisfies the two conditions
\begin{align*}
\tild{g} i &= r g' i = r cf = f  \\
q \tild{g} &= q r g' = pc r g' \\
&\simeq p g' \: \text{ rel } A \\
&= g
\end{align*}
as desired.
\end{proof}

\begin{prop}[Pulling back distributors] \label{pr:PullBackDistrib}
Let $F \colon \cat{S} \to \cat{T}$ be a morphism of left linear $\Topp_*$-categories such that for all objects $A,B$ of $\cat{S}$, the induced map $F \colon \cat{S}(A,B) \ral{\sim} \cat{T}(FA,FB)$ is a weak equivalence. %
Assume that every mapping space $\cat{S}(A,B)$ has the homotopy type of a CW complex. If $\cat{T}$ is $N$-distributive for some $N \geq 1$ (or $N = \infty$), then $\cat{S}$ is $N$-distributive.
\end{prop}

\begin{proof}
Let $\psi^N$ be an $N$-distributor for $\cat{T}$. We will prove the statement by induction, using the following condition for $n \leq N$.

\begin{itemize}
\item There is given an $n$-distributor $\phy^n$ for $\cat{S}$, based on $\phy^{n-1}$.
\item There is given a homotopy
\[
h^n \colon \psi^n F \simeq F \phy^n
\]
which is compatible with the previous steps in the following sense. For every proper subcube $C_{\si} \subset I^n$, of dimension $\dim \si = d < n$, the restriction of $h^n$ to $C_{\si}$ satisfies
\begin{equation} \label{eq:InducHyp}
h^n \vert_{C_{\si} \x I} = h^d [\si] \colon \left( \psi^d F \right) [\si] \simeq F \left( \phy^d [\si] \right).
\end{equation}
\end{itemize}
Here $\psi^n F$ denotes the collection of cubes
\[
\psi^n F = \left\{ \psi_{Fa}^{Fx_0, \ldots, Fx_n} \mid a, x_0, \ldots, x_n \in \cat{S} \right\}
\]
and $h^d [\si]$ is defined by the analogue of the formula that defines $\phy^d [\si]$, applied at each time of the homotopy.

\textbf{Base case $n=0$.} The $0$-distributor $\phy^0$ for $\cat{S}$ satisfies $F \phy^0 = \psi^0 F$, i.e., for $a,x_0 \in \cat{S}$, we have
\[
F (\phy_a^{x_0}) = F (ax_0) = (Fa)(Fx_0) = \psi_{Fa}^{Fx_0}.
\]
Take $h^0$ to be the stationary homotopy between $F \phy^0$ and $\psi^0 F$.

\textbf{Inductive step from $n-1$ to $n$.} The two composites in the square
\[
\xymatrix @C=\bigcol @R=\bigrow {
\del I^n \x \cat{S}(A,B) \x \cat{S}(X,A)^{n+1} \ar@{^{(}->}[d]_{\io} \ar[r]^-{\obs{\phy^{n-1}}} & \cat{S} (X,B) \ar[d]^{F}_{\sim} \\
I^n \x \cat{S}(A,B) \x \cat{S}(X,A)^{n+1} \ar[d]_{\id \x F \x F} \ar[r]^-{\psi^n F} & \cat{T} (FX, FB) \ultwocell<\omit>{\hspace{-2.5em}\obs{h^{n-1}}} \\
I^n \x \cat{T}(FA,FB) \x \cat{T}(FX,FA)^{n+1} \ar[ur]_-{\psi^n} & \\
}
\]
are $\psi^n F \vert_{\del I^n} = \obs{\psi^{n-1} F}$ and $F \obs{\psi^{n-1}}$. By induction hypothesis and Lemma~\ref{lem:PreserveObstruc}, the given homotopies $h^{n-1}$ define a homotopy
\[
\obs{h^{n-1}} \colon \obs{\psi^{n-1} F} \simeq F \obs{\phy^{n-1}}.
\] 
By Lemma~\ref{lem:ProdCofib}, the map $\del I^n \x \cat{S}(A,B) \x \cat{S}(X,A)^{n+1} \inj I^n \x \cat{S}(A,B) \x \cat{S}(X,A)^{n+1}$ is a Hurewicz cofibration. By the homotopy extension property, there is a homotopy
\[
\tild{h}^n \colon I^n \x \cat{S}(A,B) \x \cat{S}(X,A)^{n+1} \x I \to \cat{T}(FX,FB)
\]
extending $\obs{h^{n-1}}$ and starting at $\psi^n F$. Denote the end of the homotopy by $\tild{\psi}^n F \dfn \tild{h}^n_1$, which satisfies
\[
\tild{\psi}^n F \vert_{\del I^n} = \tild{h}^n_1 \vert_{\del I^n} = \obs{h^{n-1}}_1 = F \obs{\phy^{n-1}}.
\]
Recall that spaces of the homotopy type of a CW complex are precisely the cofibrant objects in the mixed model structure on $\Topp$. In the commutative square
\[
\xymatrix @C=\bigcol @R=\bigrow {
\del I^n \x \cat{S}(A,B) \x \cat{S}(X,A)^{n+1} \ar[d] \ar[r]^-{\obs{\phy^{n-1}}} & \cat{S}(X,B) \ar[d]^{F} \\
I^n \x \cat{S}(A,B) \x \cat{S}(X,A)^{n+1} \ar@{-->}[ur]^-{\phy^n} \ar@{}@<3.0ex>[r]^(0.8){\Uparrow \ga^n} \ar[r]^-{\tild{\psi}^n F} & \cat{T}(FX,FB), \\
}
\]
there is a map $\phy^n \colon I^n \to \cat{S}$ making the top triangle commute strictly and the bottom triangle commute up to homotopy rel $\del I^n$, by Lemma~\ref{lem:LiftWeakEq}. %
Thus $\phy^n$ is an $n$-distributor for $\cat{S}$, based on $\phy^{n-1}$. Now, define $h^n$ as the concatenation of the two homotopies
\[
\xymatrix{
\psi^n F \ar[r]^-{\simeq}_-{\tild{h}^n} & \tild{\psi}^n F \ar[r]^-{\simeq}_-{\ga^n} & F \phy^n. \\
}
\]
Since the homotopy $\ga^n$ is rel $\del I^n$ and the homotopy $\tild{h}^n$ restricts to $\obs{h^{n-1}}$ on $\del I^n$, the homotopy $h^n$ satisfies the compatibility Equation~\eqref{eq:InducHyp}, completing the inductive step.
\end{proof}

\subsection{Pushing forward distributors} \label{sec:Pushforward}

Let us recall some facts about higher associativity, which will be used later.

\begin{defn}[\cite{Stasheff70}*{Definition~8.2}]
Let $M$ and $N$ be topological monoids and $n \geq 1$ an integer. An \Def{$A_{n}$ structure} on a continuous map $f \colon M \to N$ is a family of maps $\de_i \colon I^{i-1} \x M^i \to N$ for $1 \leq i \leq n$ satisfying the following:
\begin{itemize}
\item $\de_1 = f$.
\item For $2 \leq i \leq n$, the following boundary conditions hold:
\[
\resizebox{1.0\textwidth}{!}{
$\de_i(t_1, \ldots, t_{i-1}; x_1, \ldots, x_i) = \begin{cases}
\de_{i-1}(t_1, \ldots, \widehat{t_j}, \ldots, t_{i-1}; x_1, \ldots, x_j x_{j+1}, \ldots, x_i) &\text{if } t_j = 0 \\
\de_{j}(t_1, \ldots, t_{j-1}; x_1, \ldots, x_j) \de_{i-j}(t_{j+1}, \ldots, t_{i-1}; x_{j+1}, \ldots, x_i) &\text{if } t_j=1. \\
\end{cases}$
}
\]
\end{itemize}
\end{defn}

\begin{ex}
An $A_1$ structure on $f \colon M \to N$ consists of no additional data. An $A_2$ structure consists of paths
\[
\de^{x,y} \dfn \de_2(-;x,y) \colon I \to N
\]
from $f(xy)$ to $f(x)f(y)$, depending continuously on the inputs $x,y \in M$. In other words, the map $f$ is $A_2$ if and only if it preserves the multiplication up to homotopy.

\end{ex}

\begin{rem}
As in Remark~\ref{rem:HighestDim}, an $A_n$ structure $\de = (\de_1, \ldots, \de_n)$ on a  map $f \colon M \to N$ is determined by the highest dimensional part $\de_n$, since $M$ is strictly unital. The same would not hold if $M$ were %
unital up to homotopy.
\end{rem}

Let us recast Definition~\ref{def:nDistrib} in this terminology. An $n$-distributor $\phy^n$ for a left linear $\Topp_*$-category $\cat{T}$ consists of the following data: For each $a \in \cat{T}(A,B)$, an $A_{n+1}$ structure $\phy_a$ for left multiplication by $a$, i.e., postcomposition 
\[
a_* \colon \cat{T}(X,A) \to \cat{T}(X,B),
\]
with respect to addition in the mapping spaces $\cat{T}(X,A)$ and $\cat{T}(X,B)$. These $A_{n+1}$ structures $\phy_a$ are required to depend continuously on $a \in \cat{T}$. 
Note that our indexing counts the number of plus signs in $a(x_0 + \ldots + x_n)$, which agrees with the dimension of the cube $\phy_a^{x_0, \ldots, x_n} \colon I^n \to \cat{T}(X,B)$. 

The following result is due to Fuchs \cite{Fuchs65} and can be found in \cite{Stasheff70}*{\S 8}; c.f. \cite{BoardmanV73}*{\S 4.3}.

\begin{lem} \label{lem:CompoAn}
\begin{enumerate}
\item \label{eq:CompoAn} A composition of $A_n$ maps is an $A_n$ map.
\item A map homotopic to an $A_n$ map is an $A_n$ map.
\item \label{eq:HomotInv} Let $f \colon M \ral{\simeq} N$ be a homotopy equivalence between topological monoids. Then $f$ is an $A_n$ map if and only if any homotopy inverse $g \colon N \ral{\simeq} M$ is $A_n$. 
\end{enumerate}
\end{lem}

Moreover, it follows from the explicit construction in \cite{Fuchs65} that the $A_n$ structure of a composite $gf$ depends continuously on the $A_n$ structures of $f$ and $g$.

\begin{prop}[Pushing forward distributors] \label{pr:PushForwardDistrib}
Let $F \colon \cat{S} \to \cat{T}$ be a morphism of left linear $\Topp_*$-categories such that for all objects $A,B$ of $\cat{S}$, the induced map $F \colon \cat{S}(A,B) \ral{\simeq} \cat{T}(FA,FB)$ is a homotopy equivalence. 
Assume that the functor $\pi_0 F \colon \pi_0 \cat{S} \to \pi_0 \cat{T}$ is essentially surjective. If $\cat{S}$ is $N$-distributive for some $N \geq 1$ (or $N = \infty$), then $\cat{T}$ is $N$-distributive.
\end{prop}

\begin{proof}
Consider the factorization of $F \colon \cat{S} \to \cat{T}$ as
\[
\cat{S} \to \ObIm(F) \to \cat{T},
\]
where the ``object-image'' of the functor $F$ is the full $\Topp_*$-subcategory of $\cat{T}$ consisting of objects of the form $FX$ for some object $X$ of $\cat{S}$. Then both functors $\cat{S} \to \ObIm(F)$ and $\ObIm(F) \to \cat{T}$ of this factorization satisfy the assumptions of the statement. This reduces the general statement to the following cases.
\begin{enumerate}
\item[Case (a).] $F$ is surjective on objects.
\item[Case (b).] $F$ is the identity on each mapping space, i.e., $F \colon \cat{S}(A,B) \ral{=} \cat{T}(FA,FB)$.
\end{enumerate}

\textbf{Proof for Case (a).}
For each object $A$ of $\cat{T}$, choose an object denoted $GA$ of $\cat{S}$ satisfying $FGA = A$. %
For each pair of objects $A,B$ of $\cat{T}$, choose a homotopy inverse $G \colon \cat{T}(A,B) \ral{\simeq} \cat{S}(GA,GB)$ to the map $F \colon \cat{S}(GA,GB) \ral{\simeq} \cat{T}(A,B)$, along with a homotopy $h \colon \cat{T}(A,B) \x I \to \cat{T}(A,B)$ from the identity to $FG$. Note that $G \colon \cat{T} \to \cat{S}$ is \emph{not} a functor, as it preserves neither composition nor identities $1_A \in \cat{T}(A,A)$.

Since $F$ preserves addition, it is in particular $A_{\infty}$ with respect to addition. By Lemma~\ref{lem:CompoAn}\eqref{eq:HomotInv}, $G \colon \cat{T}(A,B) \to \cat{S}(GA,GB)$ admits an $A_{\infty}$-structure with respect to addition, which we denote $\ga$. 
For any $x_0 \ldots, x_n \in \cat{T}(X,A)$, we denote by $\ga^{x_0, \ldots, x_n} \colon I^n \to \cat{S}(GX,GA)$ the corresponding $n$-cube with extreme corners $G(x_0 + \ldots + x_n)$ and $Gx_0 + \ldots + Gx_n$. 

Let $\phy^N$ be an $N$-distributor for $\cat{S}$. 
For every $a \in \cat{T}(A,B)$, consider $Ga \in \cat{S}(GA,GB)$ and the composite
\[
\xymatrix @R=0.5pc {
\cat{T}(X,A) \ar[r]^-{G} & \cat{S}(GX,GA) \ar[r]^-{(Ga)_*} & \cat{S}(GX,GB) \\
x \ar@{|->}[r] & Gx \ar@{|->}[r] & (Ga)(Gx) \\
}
\]
of $G$ and left multiplication by $Ga$, both of which are $A_{n+1}$ with respect to addition. By Lemma~\ref{lem:CompoAn}\eqref{eq:CompoAn}, this composite inherits an $A_{n+1}$-structure with respect to addition, which we denote $\xi^n$. This $A_{n+1}$ structure $\xi^n$ depends continuously on the element $Ga \in \cat{S}(GA,GB)$, and therefore on $a \in \cat{T}(A,B)$.

For instance, $\xi_a^{x,y} \colon I^1 \to \cat{S}(GX,GB)$ is the concatenation of paths:
\[
\begin{tikzpicture}[scale=2]
\draw [->-=0.5] (0,0) -- (2,0);
\draw [->-=0.5] (2,0) -- (4,0);
\draw [fill=\VertexColor] (0,0) circle (\VertexSize);
\node [below left] at (0,0) {$(Ga) G(x + y)$};
\draw [fill=\VertexColor] (2,0) circle (\VertexSize);
\node [below] at (2,0) {$(Ga) (Gx + Gy)$};
\draw [fill=\VertexColor] (4,0) circle (\VertexSize);
\node [below right] at (4,0) {$Ga Gx + Ga Gy$.};
\node [above, text=\EdgeColor] at (1,0) {$(Ga) \ga^{x,y}$};
\node [above, text=\EdgeColor] at (3,0) {$\phy_{Ga}^{Gx,Gy}$};
\end{tikzpicture}
\]

Now we prove the statement by induction, using the following condition for $n \leq N$.

\begin{itemize}
\item There is given an $n$-distributor $\psi^n$ for $\cat{T}$, based on $\psi^{n-1}$.
\item There is given a homotopy
\[
h^n \colon \psi^n \simeq F (\xi^n)
\]
which is compatible with the previous steps in the following sense. For every proper subcube $C_{\si} \subset I^n$, of dimension $\dim \si = d < n$, the restriction of $h^n$ to $C_{\si}$ satisfies
\begin{equation}\label{eq:InducHypForward}
h^n \vert_{C_{\si} \x I} = h^d [\si] \colon \left( \psi^d \right) [\si] \simeq F \left( \xi^d [\si] \right).
\end{equation}
\end{itemize}

\textbf{Base case $n=0$.} The $0$-distributor $\psi^0$ for $\cat{T}$ is forced to be $\psi_a^{x} = ax$. For the homotopy $h^0 \colon \psi^0 \simeq F (\xi^0)$, take the path
\[
\begin{tikzpicture}[scale=2]
\draw [->-=0.5] (0,0) -- (2,0);
\draw [fill=\VertexColor] (0,0) circle (\VertexSize);
\node [below left] at (0,0) {$ax$};
\draw [fill=\VertexColor] (2,0) circle (\VertexSize);
\node [below right] at (2,0) {$(FGa)(FGx)$.};
\node [above, text=\EdgeColor] at (1,0) {$(h^0)_a^x \dfn h_a h_x$};
\end{tikzpicture}
\]
where $h_a \dfn h(a,-) \colon I \to \cat{T}(A,B)$ is the path following $a$ throughout the homotopy $h \colon \cat{T}(A,B) \x I \to \cat{T}(A,B)$. %
The compatibility condition on $h^0$ is vacuous, since $I^0$ has no proper subcube.

\textbf{Inductive step from $n-1$ to $n$.} By induction hypothesis and Lemma~\ref{lem:PreserveObstruc}, the given homotopies $h^{n-1}$ define a homotopy
\[
\obs{h^{n-1}} \colon \obs{\psi^{n-1}} \simeq F \obs{\xi^{n-1}}.
\] 
By Lemma~\ref{lem:ProdCofib}, the map $\del I^n \x \cat{T}(A,B) \x \cat{T}(X,A)^{n+1} \inj I^n \x \cat{T}(A,B) \x \cat{T}(X,A)^{n+1}$ is a Hurewicz cofibration. By the homotopy extension property, there is a homotopy
\[
\tild{h}^n \colon I^n \x \cat{T}(A,B) \x \cat{T}(X,A)^{n+1} \x I \to \cat{T}(X,B)
\]
extending $\obs{h^{n-1}}$ and ending at $F (\xi^n)$. Denote the start of the homotopy by $\psi^n \dfn \tild{h}^n_0$, which satisfies
\[
\psi^n \vert_{\del I^n} = \tild{h}^n_0 \vert_{\del I^n} = \obs{h^{n-1}}_0 = \obs{\psi^{n-1}},
\]
so that $\psi^n$ is an $n$-distributor for $\cat{T}$ based on $\psi^{n-1}$. Moreover, the homotopy $\tild{h}^n \colon \psi^n \simeq F (\xi^n)$ satisfies the compatibility Equation~\eqref{eq:InducHypForward}.

\textbf{Proof for Case (b).} The functor $\pi_0 F \colon \pi_0 \cat{S} \to \pi_0 \cat{T}$ is an equivalence of categories. Choose and inverse equivalence $G \colon \pi_0 \cat{T} \to \pi_0 \cat{S}$, with a natural isomorphism $\ep \colon (\pi_0 F) G \ral{\cong} \id_{\pi_0 \cat{T}}$. For every object $X$ of $\cat{T}$, consider the inverse isomorphisms
\[
\ep_X \in (\pi_0\cat{T})(FGX,X) \quad \text{ and } \quad \ep_X^{-1} \in (\pi_0\cat{T})(X,FGX)
\]
and choose representative maps
\[
\tild{\ep_X} \in \cat{T}(FGX,X) \quad \text{ and } \quad \tild{\ep_X^{-1}} \in \cat{T}(X,FGX).
\]
By construction, these maps $\tild{\ep_X} \colon FGX \ral{\simeq} X$ and $\tild{\ep_X^{-1}} \colon X \ral{\simeq} FGX$ are inverse homotopy equivalences, and hence induce homotopy equivalences on mapping spaces, upon applying functors of the form $\cat{T}(W,-)$ or $\cat{T}(-,Z)$. Now let $X,A,B$ be objects of $\cat{T}$ and consider the diagram:
\[
\resizebox{1.0\textwidth}{!}{
\xymatrix{
\del I^n \x \cat{S}(GA,GB) \x \cat{S}(GX,GA)^{n+1} \ar@{=}[d] & & \cat{S}(GX,GB) \ar@{=}[d] \\
\del I^n \x \cat{T}(FGA,FGB) \x \cat{T}(FGX,FGA)^{n+1} \ar[ddd]^(0.6){\simeq}_(0.6){\id \x (\tild{\ep_B})_* (\tild{\ep_A^{-1}})^* \x (\tild{\ep_A})_* (\tild{\ep_X^{-1}})^*} \ar@{^{(}->}[dr] \ar[rr]^-{\obs{\phy^{n-1}}} & & \cat{T}(FGX,FGB) \ar[ddd]_(0.6){\simeq}^(0.6){(\tild{\ep_B})_* (\tild{\ep_X^{-1}})^*} \\
& \mathmakebox[7pc][r]{I^n \x \cat{T}(FGA,FGB) \x \cat{T}(FGX,FGA)^{n+1}} \ar[ddd]^(0.4){\simeq} \ar[ur]_{\phy^n} & \\
& & \\
\del I^n \x \cat{T}(A,B) \x \cat{T}(X,A)^{n+1} \ar@{^{(}->}[dr] \ar[rr]^-(0.4){\obs{\psi^{n-1}}} & & \cat{T}(X,B) \\
& \mathmakebox[5pc][r]{I^n \x \cat{T}(A,B) \x \cat{T}(X,A)^{n+1}} \ar@{-->}[ur]_{\psi^n} & \\
}
}
\]
where the back and front right faces need not commute strictly. By the same inductive argument as in Case~(a), we can push forward the $n$-distributor $\phy^n$ for $\cat{S}$ along the downward homotopy equivalences to produce an $n$-distributor $\psi^n$ for $\cat{T}$.
\end{proof}

\begin{defn}
A $\Topp$-functor $F \colon \cat{C} \to \cat{D}$ is a \Def{Dwyer--Kan equivalence} if for all objects $A,B$ of $\cat{C}$, the map $F \colon \cat{C}(A,B) \ral{\sim} \cat{D}(FA,FB)$ is a weak equivalence, and the functor $\pi_0 F \colon \pi_0 \cat{C} \ral{\simeq} \pi_0 \cat{D}$ is an equivalence of categories.
\end{defn}

\begin{cor}[Homotopy invariance] \label{cor:HomotInvar}
Let $F \colon \cat{S} \to \cat{T}$ be a morphism of left linear $\Topp_*$-categories 
which is moreover a Dwyer--Kan equivalence.

Assume that all mapping spaces in $\cat{S}$ and in $\cat{T}$ have the homotopy type of a CW complex. Then for every $n \geq 1$ (or $n = \infty$), $\cat{S}$ is $n$-distributive if and only if $\cat{T}$ is $n$-distributive.
\end{cor}

\begin{proof}
This follows from Propositions~\ref{pr:PullBackDistrib} and \ref{pr:PushForwardDistrib}.
\end{proof}

\subsection{CW-approximation}

Using homotopy invariance, we will show that the assumption of having Serre cofibrant mapping spaces in Theorem~\ref{thm:InftyDistrib} is innocuous.

\begin{lem}[CW approximation] \label{lem:CWApprox}
Let $\cat{T}$ be a weakly bilinear mapping theory. Then there is a weakly bilinear mapping theory $\cat{S}$ whose mapping spaces are CW complexes, together with a map of left linear mapping theories $\cat{S} \ral{\sim} \cat{T}$ which is a Dwyer--Kan equivalence.
\end{lem}

\begin{proof}
Denote $LX \dfn \abs{\Sing(X)}$, equipped with the counit $\ep \colon LX \ral{\sim} X$, which is a functorial CW approximation in $\Topp$, in particular a Serre cofibrant replacement. %
Define the category $\cat{S}$ with the same objects as $\cat{T}$, and mapping spaces $\cat{S}(A,B) \dfn L\cat{T}(A,B)$, along with a functor $\ep \colon \cat{S} \to \cat{T}$ defined on mapping spaces by $\ep \colon L\cat{T}(A,B) \to \cat{T}(A,B)$. %

The functor $L \colon \Topp \to \Topp$ preserves finite products, in particular the terminal object $L(\ast) = \ast$, and also satisfies $L(S^0) = S^0$. Moreover, $L$ preserves (trivial) Serre fibrations. %
From those facts, one readily checks that $\cat{S}$ has the desired properties. 
\end{proof}

\begin{prop} \label{pr:InftyDistribMixed}
Let $\cat{T}$ be a weakly bilinear mapping theory in which every mapping space $\cat{T}(A,B)$ has the homotopy type of a CW complex. Then $\cat{T}$ is $\infty$-distributive.
\end{prop}

\begin{proof}
Let $\ep \colon \cat{S} \ral{\sim} \cat{T}$ be a Serre cofibrant replacement as in Lemma~\ref{lem:CWApprox}. By Theorem~\ref{thm:InftyDistrib}, $\cat{S}$ is $\infty$-distributive. By Proposition~\ref{pr:PushForwardDistrib}, $\cat{T}$ is also $\infty$-distributive.
\end{proof}

\appendix

\section{Models for spectra} \label{sec:BFSpectra}

In this appendix, we work out some point-set features of spectra that are needed for our construction.
We first recall some properties of \emph{Bousfield--Friedlander spectra} \cite{BousfieldF78}*{\S 2.1}.

\subsection{Bousfield--Friedlander spectra}

\begin{nota}
Let $s\Set_*$ denote the category of pointed simplicial sets. Let $\Si \colon s\Set_* \to s\Set_*$ denote the reduced suspension functor, given by the smash product $\Si T = S^1 \sm T$, using the model of the circle $S^1 = \De^1 / \del \De^1$. Note that this suspension is \emph{not} the Kan suspension \cite{GoerssJ09}*{\S III.5}.

Let $\Spec$ denote the category of Bousfield--Friedlander spectra of simplicial sets \cite{BousfieldF78}*{Definition~2.1}.
\end{nota}

Equip $\Spec$ with the \emph{stable model structure}. In this model structure, a spectrum $X$ is cofibrant if and only if its bonding maps $\si^X_n \colon \Si X_n \to X_{n+1}$ are cofibrations in $s\Set$; $X$ is fibrant if and only if $X$ is an $\Om$-spectrum and levelwise fibrant. %

\begin{thm}
\cite{BousfieldF78}*{Theorem~2.3} The stable model structure makes $\Spec$ into a simplicial model category.%
\end{thm}

In particular, $\Spec$ is enriched in (pointed) simplicial sets, where the function complex $\ul{\Spec}(X,Y)$ has $n$-simplices
\[
\ul{\Spec}(X,Y)_n \cong \Hom_{\Spec} \left( X \ot \De^n, Y \right) 
\]
as described in \cite{GoerssJ09}*{\S II.2}. Via geometric realization, this yields a $\Topp_*$-category $\Spec$, with mapping spaces
\[
\Spec(X,Y) := \abs{\ul{\Spec}(X,Y)}.
\]

\begin{rem}
We could have worked with other models of spectra. 
For comparisons between different models, see \cite{BousfieldF78}*{\S 2.5}, \cite{MandellMSS01}, and \cite{Schwede07}.
\end{rem}

\subsection{Pushout-product axiom with respect to the Cartesian product}

\begin{lem} \label{lem:SimpSets}
\begin{enumerate}
\item \label{item:SuspSimpSet} For any simplicial sets $S$ and $T$, the natural map $\Si (S \x T) \to \Si S \x \Si T$ is a monomorphism. %
\item \label{item:PushoutWedge} Let $i \colon S \to S'$ and $j \colon T \to T'$ be monomorphisms of pointed simplicial sets. Then the induced map
\[
(S' \vee T') \cup_{S \vee T} (S \x T) \to S' \x T'
\]
is a monomorphism.
\end{enumerate}
\end{lem}

\begin{proof}
(1) For every $k \geq 0$, the $k$-simplices of $S^1$ are given by 
\[
(S^1)_k = \De(\ord{k},\ord{1}) / c_0 \sim c_1
\]
where $c_i$ denotes the constant function with value $i$. The $k$-simplices of the suspension $\Si T$ are
\[
(\Si T)_k = \bigvee_{i=1}^k T_k.
\]
The map $\te \colon \Si (S \x T) \to \Si S \x \Si T$ has in simplicial degree $k$ the map of pointed sets
\[
\xymatrix{
(\Si (S \x T))_k \ar@{=}[d] \ar[r]^-{\te_k} & (\Si S \x \Si T)_k \ar@{=}[d] \\
\bigvee_{i=1}^{k} (S_k \x T_k) & \left( \bigvee_{j=1}^k S_k \right) \x \left( \bigvee_{j'=1}^k T_k \right) \\ 
}
\]
which is injective.

(2) This follows from the analogous statement for pointed sets. %
\end{proof}

\begin{lem} \label{lem:WeakPushoutProd}
If $X$ and $Y$ are cofibrant spectra, then the natural map $\io \colon X \vee Y \to X \x Y$ is a cofibration. In particular, $X \x Y$ is cofibrant.%
\end{lem}

\begin{proof}
In level $0$, we have
\[
\io_0 \colon X_0 \vee Y_0 \to X_0 \x Y_0
\]
which is a cofibration in $s\Set_*$, i.e., a monomorphism. Given $n \geq 0$, consider the commutative diagram in $s\Set_*$
\[
\resizebox{\textwidth}{!}{
\xymatrix{
\Si (X \vee Y)_n \ar@{=}[d] \ar[r]^-{\Si \io_n} & \Si (X \x Y)_n \ar[d] \\
\Si X_n \vee \Si Y_n \ar[d]_{\si_n^X \vee \si_n^Y} \ar[r] & \Si X_n \x \Si Y_n \ar[d]^{\si_n^X \x \si_n^Y} \\
**[l] X_{n+1} \vee Y_{n+1} = (X \vee Y)_{n+1} \ar[r]^-{\io_{n+1}} & **[r] (X \x Y)_{n+1} = X_{n+1} \x Y_{n+1}. \\
}
}
\]
We want to show that the map of simplicial sets
\[
X_{n+1} \vee Y_{n+1} \bigcup_{\Si X_n \vee \Si Y_n} \Si (X_n \x Y_n) \to X_{n+1} \x Y_{n+1}
\]
is a monomorphism. This map is a composite
\[
\resizebox{\textwidth}{!}{
\xymatrix{
X_{n+1} \vee Y_{n+1} \bigcup_{\Si X_n \vee \Si Y_n} \Si (X_n \x Y_n) \ar[r] & X_{n+1} \vee Y_{n+1} \bigcup_{\Si X_n \vee \Si Y_n} \Si X_n \x \Si Y_n \ar[r] & X_{n+1} \x Y_{n+1}
}
}
\]
where the first step is a monomorphism, since $\Si (X_n \x Y_n) \inj \Si X_n \x \Si Y_n$ is a monomorphism, by Lemma~\ref{lem:SimpSets} \eqref{item:SuspSimpSet}. The second step is a monomorphism, by Lemma~\ref{lem:SimpSets} \eqref{item:PushoutWedge}. Finally, the map $\ast \to X \x Y$ is a composite of cofibrations
\[
\ast \inj X \inj X \vee Y \inj X \x Y
\]
and thus $X \x Y$ is cofibrant.
\end{proof}

\subsection{Construction of Eilenberg--MacLane spectra} \label{sec:EMSpectra}

There are different constructions of the Eilenberg--MacLane spectrum $HA$ as an $\Om$-spectrum, whose constituent spaces are Eilenberg--MacLane spaces $(HA)_n = K(A,n)$. The specific features of $HA$ depend on the specific construction of Eilenberg--MacLane spaces. 
The following argument using iterated classifying spaces was kindly provided to us by Marc Stephan. More details can be found in \cite{Stephan15}*{Part~I}.

Recall the following construction of the classifying space of a simplicial group; %
the topological analogue is described in \cite{May99}*{\S 16.5}.  
Consider the functor $B \colon s\Gp \to s\Set_*$ given by
\[
BG := \diag B_{\bu}(\ast,G,\ast)
\]
where $B_{\bu}(X,G,Y)$ is the two-sided bar construction \cite{May75}*{\S 7}, which is a bisimplicial set, and $\diag$ denotes its diagonal. Explicitly, $B(*,G,*)$ has in external degree $n$ the simplicial set $B_n(\ast,G,\ast) = G^n$, so that $BG$ has as $k$-simplices the set $(BG)_k = (G_k)^k$. Note that $B$ preserves finite products and thus induces a functor on abelian group objects $B \colon s\Ab \to s\Ab$.

\begin{lem} \label{lem:GoodModelHA}
Let $A$ be an abelian group. Then there is an Eilenberg--MacLane spectrum $HA$ in $\Spec$ which is fibrant, cofibrant, and an abelian group object in $\Spec$, with addition map $+ \colon HA \x HA \to HA$ compatible with the addition map of $A$.

Moreover, if $A$ is an $\F_p$ vector space, then $HA$ is an $\F_p$-vector space object in $\Spec$.
\end{lem}

\begin{proof}
Starting from an abelian group $A$, viewed as a constant simplicial abelian group, iterating the functor $B$ yields Eilenberg--MacLane spaces $B^n A \simeq K(A,n)$. Form a spectrum $HA$ defined by  $(HA)_n := B^n A$. %
The structure maps $\si_n \colon S^1 \sm B^n A \to B^{n+1} A$ have in simplicial degree $k$ the inclusion
\[
\bigvee_{i=1}^k (B^n A)_k \to \prod_{i=1}^k (B^n A)_k.
\]
In particular, $S^1 \sm B^n A \to B^{n+1} A$ is a cofibration of simplicial sets for each $n \geq 0$, so that $HA$ is cofibrant.  Moreover, $HA$ is an $\Om$-spectrum and each simplicial set $HA_n = B^n A$ is a Kan complex, since it is a simplicial group. Therefore, $HA$ is fibrant. Also, $B^n A$ is an abelian group object in $s\Set$ for each $n \geq 0$, and the structure maps $\si_n \colon S^1 \sm B^n A \to B^{n+1} A$ are linear in the factor $B^n A$, so that the adjunct structure maps
\[
\tild{\si}_n \colon B^n A \ral{\sim} \Om B^{n+1} A
\]
are maps of simplicial abelian groups.

Moreover, if $A$ is an $\F_p$-vector space, then each simplicial abelian group $B^n A$ is a simplicial $\F_p$-vector spaces, and thus $HA$ is an $\F_p$-vector space object. 
\end{proof}

\begin{rem}
It was pointed out to us by Irakli Patchkoria and Stefan Schwede that a model for the Eilenberg--MacLane spectrum $HA$ as in Lemma~\ref{lem:GoodModelHA} can also be obtained in symmetric spectra of simplicial sets, endowed with the absolute %
flat stable model structure \cite{Schwede12sym},  
also called the $S$ model structure in \cite{HoveySS00}*{Definition~5.3.6}. 
\end{rem}

Let $\sh \colon \Spec \to \Spec$ denote the \Def{shift functor} of spectra, defined by $\sh(X)_n = X_{n+1}$. The shift has the homotopy type of the suspension $\sh X \simeq \Si X$. 

\begin{cor} \label{cor:GoodModelFiniteProd}
The Eilenberg--MacLane spectrum $K_n^A := \sh^n HA \simeq \Si^n HA$ is also an $\Om$-spectrum (hence fibrant), cofibrant, and an abelian group object. Moreover, a finite product of objects $K_{n_i}^{A_i}$ is also a fibrant cofibrant abelian group object in $\Spec$. 
\end{cor}

\end{document}